\documentclass{article}
\usepackage{amsmath,amssymb,latexsym,amsthm,mathrsfs}

\newtheorem{theo}{Theorem}[section]
\newtheorem{pro}[theo]{Proposition}
\newtheorem{lem}[theo]{Lemma}

\newcommand{\ra}{\rightarrow}
\theoremstyle{definition}
\newtheorem{defin}[theo]{Definition}
\newtheorem{exa}{Example}[section]

\newtheorem{nota}{Notation}
\DeclareMathOperator*{\tra}{tr}
\theoremstyle{remark}
\newtheorem{rem}[theo]{Remark}

\begin{document}

\title{Large deviations for stable like random walks on $\mathbb{Z}^d$ with
applications to random walks on wreath products}
\author{Laurent Saloff-Coste\thanks{%
Both authors partially supported by NSF grant DMS 1004771} \\
{\small Department of Mathematics}\\
{\small Cornell University} \and Tianyi Zheng \\
{\small Department of Mathematics}\\
{\small Cornell University} }
\maketitle

\begin{abstract}
We derive Donsker-Vardhan type results for functionals of the occupation
times when the underlying random walk on $\mathbb{Z}^d$ is in the domain of
attraction of an operator-stable law on $\mathbb{R}^d$. Applications to
random walks on wreath products (also known as lamplighter groups) are given.
\end{abstract}

\section{Introduction}

\setcounter{equation}{0}

This work addresses two closely related questions of independent interests.
From the point of view of random walks on the lattices $\mathbb{Z}^d$, we
extend the well-known large deviation theorem of Donsker and Varadhan
regarding the Laplace transform of the number $D_n$ of visited points before
time $n$. The theorem of Donsker and Varadhan, \cite{Donsker1979}, treats
random walks driven by measure $\mu$ in the domain of normal attraction of a
symmetric stable law of index $\alpha \in (0,2)$ (as well as the Gaussian
case).

We generalize this result to random walks driven by a measure in the domain
of attraction of an \emph{operator-stable law}. For instance, this includes
laws that are ``stable'' with respect to anisotropic dilations of the type 
\begin{equation}  \label{Dil}
\delta_t(x_1,\dots,x_d)=(t^{1/\alpha_1}x_1,\dots,t^{1/\alpha_d} x_d)
\end{equation}
with $\alpha_i\in (0,2)$, $1\le i\le d$. In this case, the generalization of
the theorem of Donsker and Varadhan reads as follows. Let $\hat{\eta}$
denote the Fourier transform of the distribution $\eta$ on either $\mathbb{Z}%
^d$ or $\mathbb{R}^d$.

\begin{theo}
\label{DV} Referring to the anisotropic dilations at \emph{(\ref{Dil})}, assume
that $\mu$ is a symmetric measure on $\mathbb{Z}^d$ such that, uniformly on
compact sets in $\mathbb{R}^d$, 
\begin{equation*}
n[1-\hat{\mu}(\delta_n^{-1} \xi)]\rightarrow \Theta (\xi)
\end{equation*}
where $\Theta(\xi)$ as the form 
\begin{equation}  \label{Om}
\Theta(\xi)=\int_{\mathbb{S}^{d-1}}\int_0^\infty(1-\cos(\xi,\delta_r y)) 
\frac{M(dy)}{r} \frac{dr}{r}
\end{equation}
for some symmetric finite measure $M$ on $\mathbb{S}^{d-1}$ whose support
generates $\mathbb{R}^d$. Then 
\begin{equation*}
\lim_{n\rightarrow \infty}\frac{1}{n^{d/(d+\alpha)}} \log \mathbf{E }%
(e^{-\nu D_n})= k(\nu, \eta) \in (0,\infty)
\end{equation*}
where $\eta$ is the probability distribution on $\mathbb{R}^d$ such that $%
\hat{\eta}=e^{-\Theta}$ and $\frac{d}{\alpha}=\sum_1^d\frac{1}{\alpha_i}$.
\end{theo}

As in the classical Donsker-Varadhan theorem, the constant $k(\nu,\eta)$ is
described by a variational formula. In its most natural generality (see
Theorem \ref{Asymptotic} with $F(s)= \mathbf{1}_{(0,\infty)}(s)=1-%
\delta_0(s) $), this theorem involves more general dilation semigroups of
the form $t^E=\sum_0^\infty \frac{(\log t)^n E^n}{n!}$ where $E$ is an
invertible matrix with eigenvalues in $[1/2,\infty)$ ($E$ may not be
diagonalizable and, even if $E$ is diagonalizable, it may not be
diagonalizable in a basis of $\mathbb{Z}^d$ vectors). In this case, the
associated limit law $\eta$ is ``operator-stable'' with respect to the
dilation structure $t^E$, $t>0$, and the real $\alpha\in (0,2)$ is given by $%
\alpha=\mbox{tr}(E)/d$.

In fact, we are also interested in a different generalization of the
Donsker-Varadhan Theorem. Given a random walk on $\mathbb{Z}^d$ driven by a
symmetric measure $\mu$, let $l(n,x)$ denotes the number of visits at $x$ up
to time $n$. We are interested in obtaining a large deviation result for the
Laplace transform of more general functionals of the occupation time vector $%
(l(n,x))_{x\in \mathbb{Z}^d}$ than the number of visited sites, $D_n=\#\{x:
l(n,x)\neq 0\}$. For instance, we are interested in the asymptotic behavior
of 
\begin{equation*}
-\log \mathbf{E }\left(e^{-\lambda \sum_{x\in \mathbb{Z}^d} \log
l(n,x)}\right)
\end{equation*}
and, more generally, 
\begin{equation*}
-\log \mathbf{E }\left(e^{-\lambda \sum_{x\in \mathbb{Z}^d} F(l(n,x))}\right)
\end{equation*}
when $F$ belongs to some appropriate class of functions. For simplicity, in
the next theorem, we consider the case where the function $F$ is simply $%
F(s)=s^\gamma$, $\gamma\in (0,1)$ and the dilation structure is given by (%
\ref{Dil}).

\begin{theo}
\label{DVBK} Referring to the anisotropic dilations at \emph{(\ref{Dil})},
assume that $\mu$ is a symmetric measure on $\mathbb{Z}^d$ such that,
uniformly on compact sets of $\mathbb{R}^d$, 
\begin{equation*}
n[1-\hat{\mu}(\delta_n^{-1} \xi)]\rightarrow \Theta (\xi)
\end{equation*}
where $\Theta$ has the form \emph{(\ref{Om})}. Then, for $\gamma\in (0,1)$, 
\begin{equation*}
\lim_{n\rightarrow \infty}n^{-\frac{\gamma+\tau(1-\gamma)}{1+\tau(1-\gamma)}%
} \log \mathbf{E }\left(e^{-\nu \sum_x \ell(n,x)^\gamma}\right)= k(\nu,
\eta,\gamma) \in (0,\infty)
\end{equation*}
where $\eta$ is the probability distribution on $\mathbb{R}^d$ such that $%
\hat{\eta}=e^{-\Theta}$ and $\tau=\sum_1^d\frac{1}{\alpha_i}$.
\end{theo}

In the case of where $\mu$ is symmetric finitely supported and the dilation
structure is the isotropic $\delta_t(x)=\sqrt{t}x$ (in this case, the limit
law $\eta$ is Gaussian), this result is contained in \cite{Biskup2001}.
See Theorem 1.2 (with $p=1$ and $H(s)=s^\gamma$) and Section 2.3 in \cite%
{Biskup2001}.  Indeed, one of the contributions of \cite{Biskup2001}
is to show how to deduce results such as Theorem \ref{DVBK} from
the Donsker-Varadhan large deviation principle for the scaled version of the 
occupation measure of the underlying process.  
In order to treat processes that fall in the operator stable realm, we  
modify some of the arguments in \cite{Biskup2001} and rely more on the original
techniques of Donsker and Varadhan. 

The version of Theorem \ref{DVBK}
which treats dilations of the form $t^E$ and functionals $\sum_x
F(\ell(n,x)) $ associated to more general functions $F$ than power functions
is given in Theorem \ref{Asymptotic} and in Section \ref{sec-DVBK}. Except
for a few technical adaptations to the operator stable context, the proofs of
Theorem \ref{DV} is a routine generalizations of the
proof given by Donsker and Varadhan in the stable context. Similarly, the
proofs of Theorems \ref{DVBK}--Theorem \ref{Asymptotic}
involves an  adaptation of the techniques of Donsker and Varadhan and
\cite{Biskup2001}.

In developing these results in the operator stable context, we are motivated
by applications to the study of random walks on a class of groups called
wreath products. These groups are also known as lamplighter groups. The
wreath product $K\wr H$, i.e., the lamplighter group with base-group $H$ and
lamp-group $K$, will be defined precisely below. If we think of the elements
of $K$ as representing different colors (possibly countably many different
colors), then an element of $K\wr H$ can be viewed as a pair $(h,\eta)$
where $h$ is an element of $H$ (the position of the lamplighter on the base $%
H$) and $\eta=(k_h)_{h\in H}\in K^H$ is a finite configuration of colors on $%
H$ in the sense that only finitely many $h\in H$ have $k_h\neq e_K$ where $%
e_K$ is the identity element in $K$ (only finitely many lamps are turned
on). This description does not explain the group law on $K\wr H$ but
captures the nature of the elements of the wreath product $K\wr H$. The
identity element in $K\wr H$ has the lamplighter sitting at $e_H$ and all
the lamps turned off ($k_h=e_K$ for all $h\in H$). In one of the simplest
instance of this construction, $H=\mathbb{Z}$ (a doubly infinite street) and 
$K= \mathbb{Z}/2\mathbb{Z}$ (lamps are either off ($0$) or on ($1$)).

We are interested in a large collection of random walks on wreath products
which can be described collectively as the ``switch--walk--switch'' walks.
See also \cite{Pittet2002,Var}. Namely, we are given two probability
measures, one on $H$, call it $\mu$, and one on $K$, call it $\nu$. The
measure $\mu$ drives a random walk on $H$ which describes the moves of the
lamplighter (i.e., the first coordinate, $h$, in the pair $(h,\eta)\in K\wr
H $). The measure $\nu$ drives a random walk on $K$ whose basic step is
interpreted as ``switching'' between lamp colors. Based on this input, we
construct a probability measure $q=q(\mu,\nu)$ on $K\wr H$ (this measure $q$
is defined precisely later in the paper). The basic step of the walk driven
by $q$ can be accurately describes as follows: the lamplighter switches the
color of the lamp at its standing position (using $\nu$), takes a step in $H$
(using $\mu$) and switches the color of the lamp at its new position (using $%
\nu$). These different moves are, in the appropriate sense, made
independently of each other hence the name, \emph{switch--walk--switch}. Let
us insist on the fact that we will be interested here in cases when the
measures $\mu$ and $\nu$ are not necessarily finitely supported. Now, an
elementary argument shows that the probability of return $q^{(n)}(e)$ of the
random walk driven by $q$ on $K\wr H$ is given by 
\begin{equation*}
q^{(n)}(e)= \mathbf{E}\left(\prod_h \nu^{(2l_*(n,h))}(e_K) \mathbf{1}%
_{\{X_n=e_h\}}\right)
\end{equation*}
where $(X_m)_0^\infty$ is the random walk on $H$ driven by $\mu$ and $%
l_*(n,h)$ is an essentially trivial modification of the number of visits of $%
(X_m)_0^\infty$ to $h$ up to time $n$. The expectation is relative to the
random walk $(X_m)_0^\infty$ on $H$, started at $e_H$. This observation goes
back to \cite{Var} and is the basis of the analysis developed in \cite%
{Pittet2002}. If we set $F(m)=-\log \nu^{(2m)}(e_K)$ then it follows under
mild assumptions that 
\begin{equation}  \label{logsim}
\log q^{(n)}(e) \sim \log \mathbf{E}\left(e^{-\sum_h F(l(n,h))}\right).
\end{equation}
In words, the log-asymptotic of the probability of return of a
switch-walk-switch random walk on the wreath product $K\wr H$ is given by
the appropriate version of the Donsker-Varadhan large deviation theorem for
the random walk on the base $H$ driven by $\mu$. The particular functional $%
\sum_h F(l(n,h)$ that needs to be treated depends on the nature of the
lamp-group $K$ and the measure $\nu$. Formula (\ref{logsim}) is particularly
interesting because, in the general context of random walks on groups,
precise log-asymptotic of the probability of return are hard to obtain. The
following result serves to illustrate this point.

\begin{theo}[Log-asymptotics on $\mathbb{Z}^D \wr \mathbb{Z}^d$]
Fix two integers $D,d\ge 1$. Let $\nu$ be any finite symmetric measure on $%
\mathbb{Z}^D$ with $\nu(0)>0$ and generating support. Let $\delta_t$ be the
anisotropic dilation on $\mathbb{R}^d$ defined at \emph{(\ref{Dil})}. Let $%
\mu$ be a symmetric measure on $\mathbb{Z}^d$ as in \emph{Theorem \ref{DV}}
with $\delta_n^{-1}(\mu^{(n)}) \Longrightarrow \eta$ and $\hat{\eta}%
=e^{-\Theta}$. On the wreath product $\mathbb{Z}^D\wr \mathbb{Z}^d$ consider
the switch-walk-switch random walks $q=\nu*\mu*\nu$. Then 
\begin{equation*}
\lim_{n\rightarrow \infty} \frac{1}{(2n)^{\frac{d}{d+\alpha}}(\log (2n))^{%
\frac{\alpha}{d+\alpha}}}\log q^{(2n)}(e)= -c(\alpha, d,\Theta, D)
\end{equation*}%
where 
\begin{equation*}
c(\alpha,d, \Theta,D)= (D/2)^{\frac{\alpha}{d+\alpha}}\left(1+\frac{d}{\alpha%
}\right) \left( \frac{\alpha\lambda _{\Theta} }{d}\right) ^{\frac{d}{d+\alpha%
}},
\end{equation*}%
\begin{equation*}
\frac{1}{\alpha}=\frac{1}{d}\sum_1^d\frac{1}{\alpha_i} \mbox{ and }\;\;
\lambda _{\Theta} =\inf_{U: |U|=1}\{ \lambda_1(\Theta,U)\}.
\end{equation*}
Here, $\lambda_1(\Theta,U)$ is the principle eigenvalue of the infinitesimal
generator $L_\Theta$ with Dirichlet boundary condition in $U$ 
(By definition, $\widehat{L_\Theta f}= \Theta \hat{f}$).
\end{theo}

\begin{rem}
Assume that the dilations $\delta_t$ are isotropic with $\alpha_i=\alpha$, $%
i=1,\dots, d$ and that there is an Euclidean norm $\langle Q x, x\rangle$
such that $\Theta (\xi)= \langle Q\xi,\xi\rangle^\alpha$. Then $%
\lambda_\Theta $ is achieved on an Euclidean ball for the Euclidean
structure provided by $Q^{-1}$, namely, the Euclidean ball whose volume is
one. Note that the volume is computed here with respect to the Lebesgue
measure corresponding to the fixed square lattice $\mathbb{Z}^d\subset 
\mathbb{R}^d$. This fact is well-known when $\alpha=2$ and follows from \cite%
{Banuelos} when $\alpha\in (0,2)$.
\end{rem}

For any finitely generated group $G$ and any $\alpha\in (0,2)$, \cite%
{Bendikov} introduces a non-increasing function 
\begin{equation*}
\widetilde{\Phi}_{G,\rho_\alpha}: \mathbb{N}\ni n\rightarrow \widetilde{\Phi}%
_{G,\rho_\alpha}(n)\in (0,\infty)
\end{equation*}
which, by definition, provides the best possible lower bound 
\begin{equation*}
\exists\,c>0,\;N\in \mathbb{N},\;\forall\, n,\;\;\mu^{(2n)}(e)\ge c 
\widetilde{\Phi}_{G,\rho_\alpha}(Nn),
\end{equation*}
valid for every measure $\mu$ on $G$ satisfying the weak-$\alpha$-moment
condition 
\begin{equation*}
W(\rho_\alpha,\mu)=\sup_{s>0}\{ s\mu(\{g: \rho_\alpha(g)>s\})\}<\infty.
\end{equation*}
Here $|g|$ is the word-length of $G$ with respect to some fixed finite
symmetric generating set and $\rho_\alpha(g)=(1+|g|)^\alpha$. For instance,
it is well know and easy to see that 
\begin{equation*}
\widetilde{\Phi}_{\mathbb{Z}^d,\rho_\alpha}(n)\simeq n^{-d/\alpha}.
\end{equation*}
Here and throughout this paper, we write $f\sim g$ if $\lim f/g=1$ and $%
f\simeq g$ if there are constants $c_i$, $1\le i\le 4$, such that $%
c_1f(c_2t)\le g(t)\le c_3f(c_4t)$ on the relevant real interval or on $%
\mathbb{N}$. We use $\simeq $ only when at least one of the functions $f,g$
is monotone (or roughly monotone).

The main results of the present work allow us to complement some of the
lower bounds proved in \cite{Bendikov} for $\widetilde{\Phi}_{G,\rho_\alpha}$
with matching upper bounds (note that upper bounds on $\widetilde{\Phi}%
_{G,\rho_\alpha}$ are proved by exhibiting a measure with finite weak-$%
\alpha $-moment and the appropriate return probability behavior).

\begin{theo}
Fix $\alpha\in (0,2)$. Let $G$ be the group $K\wr \mathbb{Z}^d$.

\begin{enumerate}
\item Assume that $K$ is finite. Then 
\begin{equation*}
\log \widetilde{\Phi}_{G,\rho_\alpha}(n) \simeq -n^{d/(d+\alpha)}.
\end{equation*}

\item Assume that $K$ has polynomial volume growth. Then 
\begin{equation*}
\log \widetilde{\Phi}_{G,\rho_\alpha}(n) \simeq -n^{d/(d+\alpha)} (\log
n)^{\alpha/(d+\alpha)}.
\end{equation*}

\item Assume that $K$ is polycyclic with exponential volume growth. Then 
\begin{equation*}
\log \widetilde{\Phi}_{G,\rho_\alpha}(n) \simeq -n^{(d+1)/(d+1+\alpha)}.
\end{equation*}
\end{enumerate}
\end{theo}

\begin{rem}
The lower bounds are from \cite{Bendikov}. The upper bound in the first
statement is already in \cite{Bendikov} since it is based on the classical
large deviation result in \cite{Donsker1979}. The upper bounds in Statements
2 and 3 make use of the extensions of \cite{Donsker1979} in the spirit of 
\cite{Biskup2001} developed here.
\end{rem}

Iterated applications of this technique gives the following Theorem.

\begin{theo}
\label{theo-alphaiter} Fix $\alpha\in (0,2)$ and integers $d_1,\dots,d_r$.
Given a group $K$, let 
\begin{equation*}
G=(\cdots(K\wr \mathbb{Z}^{d_1})\wr \cdots )\wr Z^{d_r} \mbox{ and }
d=\sum_1^rd_i.
\end{equation*}

\begin{enumerate}
\item Assume that $K$ is finite. Then 
\begin{equation*}
-\log \widetilde{\Phi}_{G,\rho_\alpha}(n) \simeq n^{d/(d+\alpha)}.
\end{equation*}

\item Assume that $K$ has polynomial volume growth. Then 
\begin{equation*}
-\log \widetilde{\Phi}_{G,\rho_\alpha}(n) \simeq n^{d/(d+\alpha)}(\log
n)^{\alpha/(d+\alpha)}.
\end{equation*}
\end{enumerate}
\end{theo}

\section{Operator-stable laws}

\setcounter{equation}{0}

For $\alpha\in (0,2)$, the rotationally symmetric $\alpha $-stable law with
density $f_\alpha$ on $\mathbb{R}^{d}$ is the probability distribution whose
Fourier transform is $e^{-|\xi|^\alpha}$. It is embedded in a convolution
semigroup with density $f^t_\alpha$ which satisfies $f^t_\alpha(x)=
t^{-d/\alpha}f_\alpha\circ \delta^\alpha_{1/t}$ where $\delta^\alpha_t$ is
the isotropic dilation $\delta^\alpha_t(x)=t^{1/\alpha} x$, $x\in \mathbb{R}%
^d$, $t>0$.

More generally, a probability measure $\mu$ on $\mathbb{R}^d$ is called a
(non-degenerate) symmetric $\alpha$-stable law if its support is $\mathbb{R}%
^d$ and it is embedded in a probability semigroup $\mu^t$ such that $%
\delta^\alpha_t(\mu)=\mu^t$. A necessary and sufficient condition for 
that property is that $%
\hat{\mu}=e^{-\Theta}$ with 
\begin{equation*}
\Theta (\xi)= \int_{\mathbb{S}^{d-1}}\int_0^\infty(1-\cos(\xi,\delta^\alpha
_r y)) \frac{M(dy)}{r} \frac{dr}{r}
\end{equation*}
where $M$ is a finite Borel measure on $\mathbb{S}^{d-1}$ whose support
generates $\mathbb{R}^d$ (that is, the L\'evy measure $W$ of $\mu$ satisfies 
$\delta^\alpha_t(W)=tW$ and its support generates $\mathbb{R}^d$).

In the next section, we briefly review the definition of operator-stable
laws. In this definition, the role of the isotropic dilations is played by
more general one-parameter groups of transformations $t^E=\sum_0^\infty\frac{%
(\log t)^n E^n}{n!}$ where $E$ is an endomorphism of the underlying vector
space. For a detail account of the theory of operator-stable laws, see \cite%
{Hazod2001,JuMa}. Given a Borel measure $\mu$, we let $t^E(\mu)$ be the
Borel measure defined by $t^E(\mu)(A)= \mu(t^{-E}(A))$.

\subsection{Operator-stable laws}

Let $\mathbb{V}$ be a finite dimensional vector space equipped with the
Euclidean scalar product $\langle\cdot,\cdot\rangle$. Let $\mathcal{M}^{1}(%
\mathbb{V})$ denote the set of probability measures on $\mathbb{V}$. Given $%
\mu\in \mathcal{M}^1$, let $\hat{\mu}=e^{-\psi}$ denotes its Fourier
transform. Let $\mathcal{I}D(\mathbb{V})$ denotes the set of infinitely
divisible laws on $\mathbb{V}$. Throughout this section, we use notation
compatible with \cite{Hazod2001}. Recall that if $\mu\in \mathcal{I}D(%
\mathbb{V})$ with Fourier transform $e^{-\psi}$ then $e^{-t\psi}$ is the
Fourier transform of a probability measure $\mu^t$ and $(\mu^t)_{t\ge 0}$ is
a continuous convolution semigroup of measure (uniquely determined by $\mu$%
). Of course, for $\mu\in \mathcal{I}D(\mathbb{V})$, the function $\psi$
admits a Levy-Khinchine representation so that $xi \mapsto \psi(\xi)$ is the
sum of three terms, namely, the drift term $-i \langle c, \xi\rangle$ with $%
c\in \mathbb{V}$, the Gaussian term $\frac{1}{2}\langle Q \xi,\xi \rangle$,
where $Q\in \mbox{End}^+(\mathbb{V})$, and the generalized Poisson term 
\begin{equation*}
-\int_{\mathbb{V}^*} \left(e^{i\langle x,\xi\rangle}-1 -\frac{i\langle
x,\xi\rangle}{1+\|x\|^2}\right)W(dx)
\end{equation*}
where $W$ is a Levy measure. Following \cite{Hazod2001}, we call the triple $%
(c,Q,W)$ the L-K triple of $\mu$ (this triple is uniquely determined by $\mu$%
). We will be interested in the symmetric case where $c=0$ and $W(x)=W(-x)$.
In this case, the Poisson term of the Levy-Khinchine formula equals 
\begin{equation*}
\int_{\mathbb{V}^*} (1-\cos \langle x,\xi \rangle)W(dx).
\end{equation*}
In general, we let $\eta_Q$ be the (Gaussian) law associated with the triple 
$(0,Q,0)$ and $e(W)$ the (generalized-Poisson) law associated with ($0,0,W)$.

\begin{defin}[Definition 1.3.11 \protect\cite{Hazod2001}]
A law $\eta \in \mathcal{ID}(\mathbb{V})$ is said to be operator-stable if
there exist $E\in \mbox{End}(\mathbb{V})$ and a mapping $a:%
\mathbb{R}
_{+}^{\times }\rightarrow \mathbb{V}$ such that%
\begin{equation*}
t^{E}(\eta )\ast \delta _{a(t)}=\eta ^{t},
\end{equation*}%
for all $t\in 
\mathbb{R}
_{+}^{\times }.$ In this case, $E$ is called an exponent of $\eta .$ Let EXP$%
(\eta )$ denote the set of exponents of $\eta .$ If $a\equiv 0$, $\eta$ is
said to be strictly operator-stable.
\end{defin}

One can always split an operator-stable law into a Gaussian part and a
generalized Poisson part that are supported on supplementary linear
subspaces of $\mathbb{V}$.

The subspace supporting the Gaussian part is either trivial or associated
with the eigenvalues $z$ of $E$ with $\mbox{Re}(z) =1/2$ of $E$. The
subspace spanned by the support of $W$ is associated with the eigenvalues $z$
of $E$ with real part strictly larger than $1/2$. Both the Gaussian part $%
e(Q)$ and the Poisson part $e(W)$ are operator stable with exponent $E$.
Further, $T^E(W)=tW$. See the splitting theorem, \cite[Lemma 1.3.12 and
Theorem 1.3.14]{Hazod2001}.

Now we restrict our attention to symmetric operator stable laws (so that $%
c=0, W(dx)=W(-dx)$). Since $t^E(W)=tW$, the Fourier transform of $e(W)$ can
be written (with $\mathbb{S}\subset \mathbb{V}$, the unit sphere) 
\begin{equation*}
\int_{\mathbb{V}^*} (1-\cos(\langle x,\xi \rangle)W(dx) = \int_0^\infty
\int_{\mathbb{S}} (1-\cos \langle \xi , r^E y \rangle) \frac{M(dy)}{r} \frac{%
dr}{r}
\end{equation*}
where $M$ is a finite measure on $\mathbb{S}$. Compare with the hypothesis
in Theorem \ref{DV} and Theorem \ref{DVBK}.

Choose an orthonormal basis $\{e_{i}\}$ on $\mathbb{V}$ with respect to
inner product $<,>.$ The generating functional $A$ of $(\eta ^{t})_{t\geq 0}$
(see \cite[1.3.16]{Hazod2001}) is given for $f\in C^{2}(\mathbb{V})$ by 
\begin{eqnarray*}
<A,f> &=&\frac{1}{2}\sum q_{ij}\cdot \frac{\partial ^{2}}{\partial
x_{i}\partial x_{j}}f(0) \\
&&+\int_{\mathbb{V}^{\times }}\left[ f(x)-f(0)-\sum \frac{\partial }{%
\partial x_{i}}f(0)\cdot \frac{x_{i}}{1+\left\Vert x\right\Vert ^{2}}\right]
W(dx).
\end{eqnarray*}%
One can also write down the Dirichlet form of the continuous convolution
semigroup $(\eta ^{t})_{t\geq 0}$ as%
\begin{eqnarray*}
\mathcal{E}_{\eta }(f,g)&=&\frac{1}{2}\int_{ \mathbb{R} ^{d}} \sum
q_{ij}\cdot \frac{\partial f}{\partial x_{i}}(x)\frac{\partial g}{\partial
x_{j}}dx \\
&&+\frac{1}{2}\int_{ \mathbb{R} ^{d}}\int_{ \mathbb{R}
^{d}}(f(x+y)-f(x))(g(x+y)-g(x))W(dy)dx, \\
\mathcal{D}(\mathcal{E}_{\eta })&=&\{f\in L^{2}(\mathbb{V}):\mathcal{E}%
_{\eta }(f,f)<\infty \}.
\end{eqnarray*}
From the splitting theorem \cite[Theorem 1.3.14]{Hazod2001} it follows that $%
(q_{ij})$ is semi-positive definite and that the subspace where it is
positive definite is the support of the Gaussian part $e(Q)$.

\begin{exa}[Anisotropic radial operator-stable laws]
One can construct operator-stable laws with respect to non-isotropic
homogeneous norms. On $\mathbb{V}=%
\mathbb{R}
^{d},$ let $E$ be a $d\times d$ diagonal matrix with diagonal entries $%
a_{i}\in (\frac{1}{2},\infty ).$ We may assume that $a_{1}=\min_{1\leq i\leq
d}a_{i}.$ Since 
\begin{equation*}
t^{E}=\left( 
\begin{array}{cccc}
t^{a_{1}} & 0 & \dots & 0 \\ 
0 & t^{a_{2}} & \dots & 0 \\ 
\vdots & \vdots & \ddots & \vdots \\ 
0 & 0 & \cdots & t^{a_{d}}%
\end{array}%
\right) ,
\end{equation*}%
we can think of $t^{E}$ as dilations scaling differently in different
coordinates. The following norm was considered in \cite{Hebish1990}. Let $%
\mathbb{B}=\{x:\left\Vert x\right\Vert <1\}$ be the open Euclidean unit
ball, define%
\begin{equation*}
\left\Vert x\right\Vert _{\ast ,E}:=\inf \{t:t^{-a_{1}^{-1}E}x\in \mathbb{B}%
\}.
\end{equation*}%
From Theorem 1 in \cite{Hebish1990}, $\left\Vert \cdot \right\Vert _{\ast
,E} $ is a sub-additive homogeneous norm. Set 
\begin{equation*}
W(dx)=\frac{c}{ \left\Vert x\right\Vert _{\ast ,E}^{a^{-1}_{1}+\tra%
(a_1^{-1}E)}}.
\end{equation*}%
Clearly, $t^{E}(W)= t W$ for all $t\in 
\mathbb{R}
_{+}^{\times }.$ Let $\eta $ be the generalized Poisson law with L-K triple $%
(0,0,W)$. Then $\eta $ is operator-stable with exponent $E.$ Note that the
assumption $a_{1}>\frac{1}{2}$ is needed so that $W$ is a L\'{e}vy measure.
\end{exa}

\begin{exa}[Anisotropic axial operator-stable laws]
Let $E$ be as in the previous example. For $\alpha\in (0,2)$ let $\nu_\alpha$
be the one-dimensional symmetric $\alpha$-stable law (so that $\hat{\nu}%
^\alpha(y)=e^{-|y|^\alpha}$). Let $\eta$ be the product measure on $\mathbb{V%
}=\mathbb{R}^d$ given by $\eta=\otimes_1^d \nu_{1/a_i}$ so that $\hat{\eta}%
(\xi)=e^{-\sum_1^d|\xi_i|^{1/a_i}}$. Clearly, $\eta$ is operator-stable with
exponent $E$. Note that in this case, the Levy measure is supported on the
union of the axes.
\end{exa}

\subsection{Domain of operator-attraction}

For full probability laws, the class of operator-stable laws coincides with
limit distributions of normalized sums of i.i.d.\ random variables and
convergence in law of normalized sums can be characterized in terms of
convergence of Fourier transforms or convergence of generators as in
Trotter's theorem. More precisely, we have the following equivalent
characterizations of convergence.

\begin{theo}[{\protect\cite[Theorem 1.6.12 and Corollary 1.6.18]{Hazod2001}}]

\label{theo-conv} Let $\mu,\eta\in \mathcal{M}^1(\mathbb{V})$ with $\eta\in 
\mathcal{I}D(\mathbb{V})$ and $\hat{\eta}=e^{-\psi}$. Let $T_{n}\in GL(%
\mathbb{V})$ and set $\mu _{n}=T_{n}\mu$. The following properties are
equivalent:

\begin{enumerate}
\item $\mu _{n}^{(n)}\Longrightarrow \eta .$

\item $\mu _{n}^{(\left\lfloor nt\right\rfloor )}\Longrightarrow \eta ^{t}$,
uniformly in $t$ over compact subsets of $[0,\infty)$.

\item $n (1- \hat{\mu}_n)\rightarrow \psi $ uniformly on compact subsets.

\item For any $f\in C^{2}(\mathbb{V})$, $n(\mu _{n}-\delta _{0})*f(0)
\rightarrow <A,f>$ where $A$ is the generating functional of $\eta$.
\end{enumerate}
\end{theo}

Next, we introduce the definition of strict domain of operator-attraction.

\begin{defin}[Definition 1.6.3 \protect\cite{Hazod2001} ]
\bigskip Let $\eta \in \mathcal{M}^{1}(\mathbb{V}).$ Then the strict domain
of operator-attraction $\mbox{DOA}_s(\eta )$ of $\eta $ consists of all $\mu
\in \mathcal{M}^{1}(\mathbb{V})$ for which there exists a sequence $T_{n} $
in $GL(\mathbb{V})$ such that%
\begin{equation*}
\eta =\lim_{n\rightarrow \infty }T_{n}(\mu ^{(n)}).
\end{equation*}
\end{defin}

\begin{rem}
With this definition, $\mbox{DOA}_s(\eta )\neq \varnothing $ is equivalent
to saying $\eta $ can be obtained as the limiting distribution of
convolution powers of some $\mu $ after normalization (but without
re-centering). The word ``strict'' refers to the absence of re-centering.
When $T_{n}$ can be taken as the isotropic matrix $b_{n}Id,$ $b_{n}\in 
\mathbb{R}
_{+},$ this agrees with the definition of the strict domain of attraction.
\end{rem}

\begin{defin}[Definition 1.10.1 \protect\cite{Hazod2001} ]
\bigskip Let $\eta \in \mathcal{M}^{1}(\mathbb{V})$ be operator-stable. Then
its strict domain of normal operator-attraction $\mbox{DNOA}_s(\eta )$
consists of all $\mu \in \mathcal{M}^{1}(\mathbb{V})$ such that%
\begin{equation*}
\eta =\lim_{n\rightarrow \infty }n^{-E}(\mu ^{(n)})
\end{equation*}%
for some $E\in \mbox{EXP}(\eta ).$
\end{defin}

\begin{exa}
\bigskip Let $U$ (resp. $V$) be a random variable on $%
\mathbb{Z}
$ in the domain of normal attraction of the $\alpha$ (resp. $\beta $)
symmetric-stable law $\nu _{\alpha }$ ($\nu _{\beta }$ resp.) on $%
\mathbb{R}
$. The measure $\eta $ on $%
\mathbb{R}
^{2}$ $\eta (dx,dy)=\nu _{\alpha }(dx)\otimes \nu _{\beta }(dy) $ is
operator stable with exponent 
\begin{equation*}
E=\left( 
\begin{array}{cc}
\frac{1}{\alpha } & 0 \\ 
0 & \frac{1}{\beta }%
\end{array}%
\right) .
\end{equation*}
It is clear that the law of $(U,V)^{T}$ is in $\mbox{DNOA}_{s}(\eta ).$ Set%
\begin{equation*}
\left( 
\begin{array}{c}
X \\ 
Y%
\end{array}%
\right) =\left( 
\begin{array}{cc}
\cos \theta & -\sin \theta \\ 
\sin \theta & \cos \theta%
\end{array}%
\right) \left( 
\begin{array}{c}
U \\ 
V%
\end{array}%
\right) ,
\end{equation*}%
and let $\mu $ denote the distribution of $(X,Y)^{T}.$ In order to obtain
convergence of $\mu^{(n)}$ we need to rotate back by a rotation of angle $%
\theta $ then normalize component-wise. That is, setting 
\begin{equation*}
T_{n}=\left( 
\begin{array}{cc}
n^{-\frac{1}{\alpha }} & 0 \\ 
0 & n^{-\frac{1}{\beta }}%
\end{array}%
\right) \left( 
\begin{array}{cc}
\cos \theta & \sin \theta \\ 
-\sin \theta & \cos \theta%
\end{array}%
\right) ,
\end{equation*}%
we have $\eta =\lim_{n\rightarrow \infty }T_{n}(\mu ^{(n)}).$ So, in this
case, $\mu\in \mbox{DOA}_s(\eta)$ but does not belong to $\mbox{DNOA}%
_s(\eta) $.
\end{exa}

\begin{rem}
Theorem 4.11.5 of \cite{JuMa} gives a practical criterion to show that a
given measure $\nu$ belongs to the domain of normal attraction of a full
operator-stable law $\eta$ without Gaussian part. More precisely, the
following statement is a simple modification of \cite[Theorem 4.11.5]{JuMa}.
Let $\eta$ be operator-stable, symmetric, with no Gaussian part and L\'evy
measure $W$ given by 
\begin{equation*}
W(B)=\int_{\mathbb{S}}\int_0^\infty \mathbf{1}_B(r^Ey) \frac{M(dy)}{r}\frac{%
dr}{r}.
\end{equation*}
Let $E\in \mbox{EXP}(\eta ).$ A necessary and sufficient condition for a
probability measure $\nu$ to be such that $n^{-E}(\nu^{(n)})\Longrightarrow
\eta $ is that 
\begin{equation*}
\lim_{t\rightarrow \infty} t\nu(\{ s^E x: x\in \Omega,\; s>t\})=M(\Omega)
\end{equation*}
for any measurable $\Omega \in \mathbb{S}_d$ with $M(\partial \Omega)=0$.
\end{rem}

\begin{exa}
For $\gamma\in (0,2)$, let $\eta_\gamma^t$ be symmetric stable law on $%
\mathbb{R}$ with Fourier transform $e^{-t|\xi|^\gamma}$. On $\mathbb{Z}$,
fix a doubly infinite symmetric sequence $z_k=-z_{-k}$, $k\in \mathbb{Z}$,
and reals $p_k=p_{-k}\ge 0$ with $\sum p_k=1$. Consider the probability
measure 
\begin{equation*}
\mu=\sum_{k\in \mathbb{Z}}p_k\mathbf{1}_{z_k}.
\end{equation*}
Consider the case when $z_k= \lfloor k^\beta \rfloor$ and $p_k= c_\alpha (1+
|k|)^{-\alpha}$ with $\alpha >1$, $\beta\ge 1$ and $\gamma=
\beta/(\alpha-1)<2$. Then, by Remark 2.6 (in fact, in this particular case,
by \cite[Theorem 4.11.5]{JuMa}), $n^{-\gamma} \mu^{(n)} \Longrightarrow
\eta_{\gamma}^c$ for some fixed $c>0$.

Note that if $z_k=\lfloor 2^{\beta k}\rfloor$ and $p_k=2^{-\alpha k}$ with $%
\alpha,\beta>0$ and $\gamma=\beta/\alpha$ then $t\nu(\{s^\gamma: s>t\}) $
stays in a compact interval in $(0, \infty)$ but does not converges.
\end{exa}

\begin{nota}
\bigskip A measure $\mu \in \mathcal{M}^{1}(\mathbb{V})$ is said to be
adapted if $\mu $ is not supported by a proper linear subspace of $\ \mathbb{%
V}$. Let $\mathcal{M}_{a}^{1}(\mathbb{V})$ denote the set of adapted
probability measures on $\mathbb{V}.$
\end{nota}

The theorem below is a characterization of strictly operator-stable laws as
those adapted distributions whose domain of strict operator-attraction is
non-empty.

\begin{theo}[Theorem 1.6.4 \protect\cite{Hazod2001} ]
For $\eta \in \mathcal{M}_{a}^{1}(\mathbb{V})$ the following assertions are
equivalent:

\begin{enumerate}
\item $\eta $ is strictly operator-stable.

\item $\eta \in \mbox{\em DOA}_s(\eta ).$

\item $\mbox{\em DOA}_s(\eta )\neq \varnothing .$
\end{enumerate}
\end{theo}

\bigskip

For $\mu \in \mbox{DOA}_s(\eta ),$ the choice of normalization sequence $%
T_{n}$ is in general not unique. In particular, we can adjust $T_{n}$ using
the symmetries of the limiting distribution $\eta $ and the convergence
still holds.

\begin{defin}[Definition 1.2.8. \protect\cite{Hazod2001} ]
Let $\eta \in \mathcal{M}^{1}(\mathbb{V})$ be non-degenerate. Let $\mbox{Sym}%
(\eta )$ be the set of all $A\in GL(\mathbb{V})$ such that there exists some 
$a\in \mathbb{V}$ such that $A(\eta )\ast \delta _{a}=\eta .$ The group $%
\mbox{Sym}(\eta )$ is called the symmetry group of $\eta $. It is a closed
subgroup of $GL(\mathbb{V}).$ The invariance group $\mbox{Inv}(\eta )$ is
the set of all $A\in GL(\mathbb{V})$ such that $A(\eta )=\eta $. The group $%
\mbox{Inv}(\eta )$ is a closed subgroup of $\mbox{Sym}(\eta ).$
\end{defin}

\bigskip

The following technical result is important for our purpose. It says that we
can always adjust the normalization sequence by elements in $\mbox{Inv}(\eta
),$ so that the new normalization sequence has nice regular variation
properties.

\begin{theo}[Theorem 1.10.19 \protect\cite{Hazod2001} ]
\label{RegularVariation} Suppose $\mu $ is in the strict domain of
attraction of a full operator stable law $\eta ,$ that is, there exists a
sequence of invertible matrices $B_{n}\in GL(\mathbb{V})$ such that 
\begin{equation*}
B_{n}^{-1}\mu ^{(n)}\Longrightarrow \eta .
\end{equation*}%
Then there exists a modified normalization sequence $\{B_{n}^{\prime
}=B_{n}S_{n}\}$, $S_{n}\in \mbox{\em Inv}(\eta ),$ hence still fulfilling%
\begin{equation*}
(B_{n}^{\prime })^{-1}\mu ^{(n)}\Longrightarrow \eta ,
\end{equation*}%
with the property that $\{B_{n}^{\prime }\}$ has regular variation in the
sense that 
\begin{equation*}
B_{n}^{\prime }(B_{\left\lfloor nt\right\rfloor }^{\prime })^{-1}\rightarrow
t^{-E},
\end{equation*}%
where the convergence is uniform in $t$ on compact subsets of $%
\mathbb{R}
_{+}^{\times }.$
\end{theo}

\subsection{\textbf{Two} more \textbf{examples on }$%
\mathbb{Z}
^{2}$}

In this subsection we discuss two examples on $%
\mathbb{Z}
^{2}$ that are in the strict domain of operator-attraction of some
operator-stable laws. For later use, we include the additional requirement
that the inverse of the normalization sequence preserve the lattice $%
\mathbb{Z}
^{2}.$

Note that a key point in these examples is that they describe probability
measures supported on the square lattice $\mathbb{Z}^2\subset \mathbb{R}^2$
which implies a certain rigidity in the choice of the Euclidean structure on 
$\mathbb{R}^2$.

\begin{exa}
Let $e_{1}$, $e_{2}$ be the standard basis for $%
\mathbb{R}
^{2}.$ Consider the two unit vectors $u_{1}=\frac{1}{\sqrt{2}}(e_{1}+e_{2})$
and $u_{2}=\frac{1}{\sqrt{1+\pi ^{2}}}(e_{1}+\pi e_{2})$. Let $\mu _{1},\mu
_{2}$ be probability measures defined by 
\begin{equation*}
\mu _{1}(x_{1},x_{2})=\frac{c_{1}}{(1+\left\vert x_{1}\right\vert )^{\alpha
+1}}\boldsymbol{1}_{\{x_{1}=x_{2}\}},
\end{equation*}%
\begin{equation*}
\mu _{2}(x_{1},x_{2})=\frac{c_{2}}{(1+\left\vert x_{1}\right\vert )^{\beta
+1}}\boldsymbol{1}_{\{\left\vert x_{2}-\pi x_{1}\right\vert \leq 1\}},
\end{equation*}%
where $c_{1}$ and $c_{2}$ are normalizing constants and $\alpha ,\beta \in
(0,2) $. Take 
\begin{equation*}
\mu =\frac{1}{2}\left(\mu _{1}+\mu _{2}\right).
\end{equation*}%
Write $P=(u_{1},u_{2})=\left( 
\begin{array}{cc}
\frac{1}{\sqrt{2}} & \frac{1}{\sqrt{1+\pi ^{2}}} \\ 
\frac{1}{\sqrt{2}} & \frac{\pi }{\sqrt{1+\pi ^{2}}}%
\end{array}%
\right) $ and $E=P\left( 
\begin{array}{cc}
\frac{1}{\alpha } & 0 \\ 
0 & \frac{1}{\beta }%
\end{array}%
\right) P^{-1}. $ Then we can check that for $\Omega \in \mathcal{B}(\mathbb{%
S}_{2})$ with $u_1,u_2\not\in \partial \Omega$, we have 
\begin{equation*}
\lim_{t\rightarrow \infty }t\mu (\{s^{E}x:x\in \Omega ,\;s>t\})=\lambda _{1}%
\boldsymbol{1}_{\left\{ u_{1}\right\} }(\Omega )+\lambda _{2}\boldsymbol{1}%
_{\left\{ u_{2}\right\} }(\Omega ),
\end{equation*}%
where $\lambda _{1}$ and $\lambda _{2}$ are positive constants determined by 
$\mu $. Consider the generalized Poisson law $\eta $ with L\'{e}vy measure $%
W $ given by 
\begin{equation*}
W(B)=\sum_{i=1,2}\int_{0}^{\infty }\mathbf{1}_{B}(r^{E}u_{i})\frac{\lambda
_{i}}{r}\frac{dr}{r}.
\end{equation*}
Note that $W$ is supported on the union of the one-dimensional subspaces $%
\mathbb{R}
u_{1}$ and $%
\mathbb{R}
u_{2}$ and $\eta $ is operator-stable with exponent $E$. The law $\eta$ can
be viewed as a product of two one-dimensional symmetric stable laws
supported on $%
\mathbb{R}
u_{1}$ and $%
\mathbb{R}
u_2$ and of exponents $\alpha$ and $\beta$, respectively (the exact scale
parameter for each of these stable laws is determined by the constants $%
\lambda_1,\lambda_2$). From the convergence theorem, Theorem \ref{theo-conv}%
, $\mu \in \mbox{DNOA}_{s}(\eta ).$ Set 
\begin{equation*}
B_{n}=\left\lfloor P\left( 
\begin{array}{cc}
n^{\frac{1}{\alpha }} & 0 \\ 
0 & n^{\frac{1}{\beta }}%
\end{array}%
\right) P^{-1}\right\rfloor
\end{equation*}
where $\left\lfloor \cdot \right\rfloor $ means take integer parts of each
matrix entry. Then $n^{E}-B_{n}$ is a matrix with entries in $[0,1)$ and it
follows that $B_{n}\cdot n^{-E}\rightarrow I.$ We conclude that $%
B_{n}^{-1}\mu ^{(n)}\Longrightarrow \eta $.
\end{exa}

\bigskip

\begin{exa}
Take $u_{1},u_{2}, P,\alpha ,\beta $ and $E$ the same as in the previous
example. Write $\delta _{r}(x_{1},x_{2})=(r^{1/\alpha }x_{1},r^{1/\beta
}x_{2})$. Let $\mathbb{S}$ denote the Euclidean unit circle, $\sigma $ the
Lebesgue measure on $\mathbb{S}$. Let $\Gamma $ be the union of the two arcs 
$[0,\frac{\pi }{2}]$ and $[\pi ,\frac{3\pi }{2}]$ of $\mathbb{S}$. Define
the L\'{e}vy measure $W$ as 
\begin{equation*}
W(B)=\int_{\Gamma }\int_{0}^{\infty }\mathbf{1}_{B}(\delta _{r}x)\frac{%
\sigma (dx)}{r}\frac{dr}{r}.
\end{equation*}%
Then $W$ is supported in the cone $\{x\in 
\mathbb{R}
^{2}:x_{1}x_{2}\geq 0\}.$ Consider measure $PW(\cdot )=W(P^{-1}\cdot ),$
that is the pushforward of measure $W$ under linear transformation $P.$ Let $%
\eta $ be the generalized Poisson law with L\'{e}vy measure $PW.$ Take a
discrete approximation $\mu $ of $PW$ supported on $%
\mathbb{Z}
^{2}$ by setting 
\begin{equation*}
\mu (x)=PW([x_{1},x_{1}+1)\times \lbrack x_{2},x_{2}+1)).
\end{equation*}%
One can check that $nB_{n}^{-1}\mu \rightarrow PW$ weakly with $%
B_{n}=\left\lfloor n^{E}\right\rfloor .$ It follows from Theorem \ref%
{theo-conv} that 
\begin{equation*}
B_{n}^{-1}\mu ^{(n)}\Longrightarrow \eta .
\end{equation*}
\end{exa}

\section{Functionals of the occupation time vector}

\setcounter{equation}{0}

Given a probability measure $\mu$ on the lattice $\mathbb{Z}^d$, let $%
(X_i)_0^\infty$ be the associated random walk. Let $(l(n,x))_{x\in \mathbb{Z}%
^d}$ be the occupation time vector at time $n$ where $l(n,x)=\#\{k\in
\{0,\dots n\}: X_k=x\}$. Let $F:[0,\infty)\rightarrow [0,\infty)$.

In this section we introduce basic natural hypotheses on $\mu$ and $F$ under
which we can derive the log-asymptotic behavior of 
\begin{equation*}
\mathbf{E }\left(e^{-\sum_{x\in\mathbb{Z}^d}F(l(n,x))}\right).
\end{equation*}

\begin{defin}[Convergence assumption]
\label{def-conv} We say that $\mu $ satisfies the convergence assumption (C-$%
B_n$) if there exists a sequence of invertible matrices $B_{n}\in \mathbb{Z}%
^{d\times d}$ and a probability distribution $\eta$ such that 
\begin{equation}
B_{n}^{-1}\mu ^{(n)}\Longrightarrow \eta .  \tag{C-$B_{n}$}
\label{ConvAssum}
\end{equation}
\end{defin}

\begin{rem}
Note that (C-$B_n$) requires the matrices $B_n$ to have integer entries so
that $B_n\mathbb{Z}^d\subset \mathbb{Z}^d$. Note also that the distribution $%
\eta$ is strictly operator-stable.
\end{rem}

Under the convergence assumption (C-$B_n$), \cite{Griffin1986} provides a
local limit theorem that plays an important role in the proof of the uniform
large deviation principle.

\begin{theo}[Theorem 6.4 \protect\cite{Griffin1986}]
\label{LocalLimit}Suppose $\mu $ is in the domain of attraction of a
symmetric, adapted strictly operator-stable law $\eta $ on $%
\mathbb{R}
^{d}$ with density $g$, that is, there exists a sequence of invertible
matrices $B_{n}$ such that 
\begin{equation*}
B_{n}^{-1}\mu ^{(n)}\Longrightarrow \eta .
\end{equation*}%
Then 
\begin{equation*}
\lim \sup_{n\rightarrow \infty }\sup_{x\in 
\mathbb{Z}
^{d}}\left\vert \det B_{n}\right\vert \left\vert \mu ^{(n)}(x)-\left\vert
\det B_{n}^{-1}\right\vert g(B_{n}^{-1}x)\right\vert =0.
\end{equation*}
\end{theo}

\begin{rem}
Note that the density $g$ of an operator-stable law is always smooth. In 
\cite{Donsker1979} it is essentially proved, although not stated explicitly,
that given the local limit theorem, the scaled occupation time measures
satisfy a uniform large deviation principle in $\mathcal{L}_1.$ We will
state and outline the proof of the large deviation principles later in this
paper.
\end{rem}

\begin{rem}
\bigskip It is somewhat surprising that, in this case, the ``weak limit
assumption'' always implies the local limit theorem. The proof in \cite%
{Griffin1986} relies on the Fourier transform. On the Heisenberg group,
there are measures that converges to a (Heisenberg group) Gaussian law, but
do not satisfy the local limit theorem.
\end{rem}

\begin{exa}
Fix $\alpha,\beta\in (0,2)$ and consider the probability measure $\mu$ on $%
\mathbb{Z}^2\subset \mathbb{R}^2$ given by 
\begin{equation*}
\mu= \frac{1}{2}\left(\sum_{x\in \mathbb{Z}} c_\alpha (1+|x|)^{-1-\alpha} 
\delta_{(x,0)}+ \sum_{x\in \mathbb{Z}} c_\beta (1+|y|)^{-1-\beta}
\delta_{(0,y)}\right) .
\end{equation*}
Set $E=\left(%
\begin{array}{cc}
1/\alpha & 0 \\ 
0 & 1/\beta%
\end{array}%
\right)$ and $\eta= \eta_{\alpha} \otimes \eta_\beta$ where $\eta_\alpha,
\eta_\beta$ are (appropriately scaled) one dimensional symmetric stable laws
with parameters $\alpha$, $\beta$, respectively. Then condition (C-$B_n$)
is satisfied with $B_n= \left(%
\begin{array}{cc}
\lfloor n^{1/\alpha} \rfloor & 0 \\ 
0 & \lfloor n^{ 1/\beta}\rfloor%
\end{array}%
\right)$. Theorem \ref{LocalLimit} provides a local limit theorem for $%
\mu^{(n)}((x,y))$ in the form 
\begin{equation*}
n^{1/\alpha+1/\beta}|\mu^{(n)}((x,y))-
n^{-(1/\alpha+1/\beta)}f^{c_1}_\alpha(x/n^{1/\alpha})f^{c_2}_\beta(y/n^{1/%
\beta})|\rightarrow 0
\end{equation*}
where $f^t_\alpha$ is the density of the symmetric stable semigroup, i.e.,
has Fourier transform $e^{-t|\xi|^\alpha}$ and $c_1,c_2$ are appropriate
constants.
\end{exa}

Next we introduce a scaling assumption regarding the function $F$. It is the
operator-stable analog of the scaling assumption in \cite{Biskup2001}.

\begin{defin}[Scaling assumption]
\label{def-scal} Let $B_n$ be as in condition (C-$B_n$). We say that a
function $F:[0,\infty )\rightarrow \lbrack 0,\infty )$ satisfies the scaling
assumption (S-$B_n$-$a_n$) if $F$ is concave, sub-additive, increasing with $%
F(0)=0$ and there exist a non-decreasing sequence $n\rightarrow a_{n}\in 
\mathbb{N}$ and a limiting function $\widetilde{F}:[0,\infty )\rightarrow
\lbrack 0,\infty ),$ $\widetilde{F}$ not identically zero, such that for $%
y>0,$%
\begin{equation}
\lim_{n\rightarrow \infty }\frac{a_{n}\det (B_{a_{n}})}{n}F\left( \frac{n}{%
\det (B_{a_{n}})}y\right) =\widetilde{F}(y),  \tag{S-$B_{n}$-$a_n$}
\end{equation}%
uniformly over compact sets in $(0,\infty ).$
\end{defin}

The following technical proposition is crucial. It is analogous to \cite[%
Proposition 1.1]{Biskup2001}. The proof is given in the Appendix.

\begin{pro}
\label{Homogenous}Assume the convergence assumption \emph{(C-$B_{n}$)} and
the scaling assumption \emph{(S-$B_{n}$-$a_{n}$)} as above. Then there
exists $\gamma \in \lbrack 0,1]$ such that 
\begin{equation*}
\widetilde{F}(y)=\widetilde{F}(1)y^{\gamma },\text{ }y>0,
\end{equation*}%
Moreover, there exists $\kappa>0$ such that 
\begin{equation*}
\lim_{n\rightarrow \infty }\frac{a_{\left\lfloor \lambda n\right\rfloor }}{%
a_{n}}=\lambda ^{\kappa }\text{ for all }\lambda \in 
\mathbb{R}
^{+} \mbox{
and } \lim_{n\rightarrow \infty }\frac{\log a_{n}}{\log n}=\kappa .
\end{equation*}
\end{pro}

\begin{defin}
Following \cite{Biskup2001}, given $\mu$ satisfying (C-$B_n$) and a function 
$F:[0,\infty)\rightarrow [0,\infty)$, we say that the pair $(F,(B_{n}))$ is
in the $\gamma $-class, if there is a sequence $a_{n}$ such that the scaling
assumption (S-$B_{n}$-$a_{n}$) is satisfied, and the limiting function $%
\widetilde{F}$ is homogeneous with exponent $\gamma .$
\end{defin}

The following statement is the main result of this paper. 
The proof is given in Section \ref{sec-DVBK}. 

\begin{theo}
\label{Asymptotic} Fix a symmetric probability measure $\mu$ on $\mathbb{Z}%
^d $ and a function $F:[0,\infty)\rightarrow [0,\infty)$. Under the
convergence assumption \emph{(C-$B_{n}$)} and the scaling assumption \emph{%
(S-$B_{n}$-$a_{n}$)}, there exists a constant $k(\eta,\widetilde{F})\in
(0,\infty)$ such that 
\begin{equation}  \label{asy1}
\lim_{n\rightarrow \infty }\frac{a_{n}}{n}\log \mathbf{E}\left(
e^{-\sum_{x\in 
\mathbb{Z}
^{d}}F(l(n,x))}\right) =-k\left( \eta ,\widetilde{F}\right) .
\end{equation}
Further, for any $\epsilon>0$ small enough there is $R>1$ such that 
\begin{equation}  \label{asy2}
\lim_{n\rightarrow \infty }\frac{a_{n}}{n}\log \mathbf{E}\left( e^{
-\sum_{x\in 
\mathbb{Z}
^{d}}F(l(n,x))} \mathbf{1}_{B(R)}(B^{-1}_{a_n}(X_n)) \right) \ge
-(1+\epsilon) k\left( \eta ,\widetilde{F}\right) .
\end{equation}
Here $B(R)$ is the ball of radius $R$ in $\mathbb{R}^d$.
\end{theo}

\begin{exa}
Theorems \ref{DV} and \ref{DVBK} are special cases of Theorem \ref%
{Asymptotic}. In both cases, let $E$ be the diagonal matrix with $i$-th
diagonal entry $1/\alpha_i\in (2,\infty)$. Let $\eta$ be an operator-stable
law with exponent $E$ and Fourier transform $\hat{\eta}=e^{-\Theta}$. Let $%
\mu$ be a measure such that 
\begin{equation}  \label{DVhyp}
n(1-\hat{\mu}(n^{-E}\xi))\rightarrow \Theta (\xi).
\end{equation}
Let $B_n$ be the diagonal matrix with $i$-th diagonal entry $\lfloor
n^{1/\alpha_i} \rfloor$. By Theorem \ref{theo-conv}, (\ref{DVhyp}) implies
that condition (C-$B_n$) is satisfied.

To obtain Theorem \ref{DV}, set $F(s)=\mathbf{1}_{(0,\infty)}(s)$. Define $%
a_n= \lfloor a^{\prime }_n\rfloor$ where $a^{\prime }_n$ is given by $%
a^{\prime }_n \mbox{det}(B_{a^{\prime }_n})=n$, that is, $a^{\prime }_n=
n^{1/(1+\tau)}$ where $\tau=\sum 1/\alpha_i$ is the trace of $E$. It is easy
to see that condition (S-$B_n$-$a_n$) with $\widetilde{F}=\mathbf{1}%
_{(0,\infty)}$.

For Theorem \ref{DVBK}, we simply set $F(s)=\widetilde{F}(s)=s^\gamma$, $%
\gamma\in (0,1)$ and $a^{\prime }_n= n^{(1-\gamma)/(1+\tau(1-\gamma))}$.
Condition (S-$B_n$-$a_n$) follows.
\end{exa}

\begin{exa}
\label{exa31} Assume that $\mu\in \mbox{DNOA}_s(\eta)$, $B_n= \lfloor
n^E\rfloor$, $\tra(E)=\tau$ and $F(y)= y^\gamma \ell( y)$ where $\gamma\in
[0,1]$ and $\ell$ is a slow varying function (at infinity) such that $%
\ell(t^a\ell(t)^b) \sim c(a) \ell(t)$ for any $a>0$ and $b\in \mathbb{R}$.
(e.g., $\ell(t)= (\log t)^\beta $, $\beta\in \mathbb{R}$). Then $\widetilde{F%
}(y)=cy^\gamma$ and $a_n$ is determined by solving 
\begin{equation*}
a_n^{1+\tau(1-\gamma)} \ell( n a_n^{-\tau})\sim n^{1-\gamma},
\end{equation*}
that is 
\begin{equation*}
a_n \sim c \left(\frac{n^{1-\gamma}}{\ell(n)}\right)^{1/(1+\tau(1-\gamma))}.
\end{equation*}
In this case the theorem yields the existence of a constant $k\in (0,\infty)$
such that 
\begin{equation*}
\log \mathbf{E}\left( \exp \left( -\sum_{x\in \mathbb{Z} ^{d}}F(l(n,x))%
\right) \right) \sim -k \left(n^{\gamma +\tau(1-\gamma)}
\ell(n)\right)^{1/(1+\tau(1-\gamma))}.
\end{equation*}
\end{exa}

\begin{exa}The previous examples treat cases where $\mu$ belongs to the 
domain of {\em normal} attraction of $\eta$. 
It is worth pointing out that Theorem \ref{Asymptotic} does not require 
normal attraction. For example, consider the case where
$\mu$ is supported on $\mathbb Z$ and is of the form 
$\mu(k)= \frac{c_\phi}{\phi(|k|)}$ where $\phi: [0,\infty)\ra [1,\infty)$ 
is continuous and  
regularly varying of index $1+\alpha$, $\alpha\in (0,2)$.
By a classical result (see \cite{FelB5}), 
$\mu $ is in the domain of attraction of an $\alpha$-stable law $\eta$ on 
$\mathbb{R}$. The normalizing sequence $b_{n}$ such that $b_n^{-1}(\mu^{(n)})
\Longrightarrow \eta $
can be chosen as the solution
of the equation $b_{n}^{-2}\mathcal{G}(b_{n})=1/n$ where 
$\mathcal G(n)=\sum_0^n k^2\mu(|k|)$, that is $b_{n}\sim \kappa \psi(1/n)$
where $\psi$ is the inverse of $s\mapsto s/\phi(s)$. 
Note that $\psi$ is regularly varying of index $-1/\alpha$.  
Suppose that $F(s)=s^\gamma$, $\gamma\in (0,1)$.  The sequence $a_n$ 
in Theorem \ref{Asymptotic} is then given by equation (S-$B_n$-$a_n$), that is,
$ n^{-1} a_n b_{a_n}(n/b_{a_n})^\gamma =1$, equivalently,
$ a_n \psi(1/a_n)^{1-\gamma}= \kappa^{\gamma-1} n^{1-\gamma}.$ It follows that $a_n$  varies regularly of index $\alpha(1-\gamma)/(1+\alpha-\gamma)$. Of course,
$a_n$ can be computed more explicitly in terms of $\phi$.
\end{exa}

\section{Applications to random walks on groups}

\setcounter{equation}{0}

This section applies the large deviation asymptotics of Theorem \ref%
{Asymptotic} to obtain precise information about the decay of the return
probability of random walks on wreath products with base $\mathbb{Z}^{d}.$
We treat certain classes of random walks with unbounded support on the base
and we allow a large class of lamp groups.

\subsection{Random walks on wreath products}

First we briefly review definition of wreath products and a special type of
random walks on them. Our notation follows \cite{Pittet2002}. Let $H$, $K$
be two finitely generated groups. Denote the identity element of \ $K$ by $%
e_K$ and identity element of $H$ by $e_H$ Let $K_{H}$ denote the direct sum:%
\begin{equation*}
K_{H}=\sum_{h\in H}K_{h}.
\end{equation*}%
The elements of $K_{H}$ are functions $f:H\rightarrow K$, $h\mapsto f(h)=k_h$%
, which have finite support in the sense that $\{h\in H: f(h)=k_h\neq e_K\}$
is finite. Multiplication on $K_H$ is simply coordinate-wise multiplication.
The identity element of $K_H$ is the constant function $\boldsymbol{e}%
_K:h\mapsto e_K$ which, abusing notation, we denote by $e_K$. The group $H$
acts on $K_{H}$ by translation:%
\begin{equation*}
\tau _{h}f(h^{\prime -1}h^{\prime }),\;\;h,h^{\prime }\in H.
\end{equation*}%
The wreath product $K\wr H$ is defined to be semidirect product%
\begin{equation*}
K\wr H=K_{H}\rtimes _{\tau }H,
\end{equation*}%
\begin{equation*}
(f,h)(f^{\prime },h^{\prime })=(f\cdot \tau _{h}f^{\prime },hh^{\prime }).
\end{equation*}%
In the lamplighter interpretation of wreath products, $H$ corresponds to the
base on which the lamplighter lives and $K$ corresponds to the lamp. We
embed $K$ and $H$ naturally in $K\wr H$ via the injective homomorphisms%
\begin{eqnarray*}
k &\longmapsto &\underline{k}=(\boldsymbol{k}_{e_H},e_H), \;\;\boldsymbol{k}%
_{e_H}(e_H)= k,\; \boldsymbol{k}_{e_H}(h)=e_K \mbox{ if } h\neq e_H \\
h &\longmapsto &\underline{h}=(\boldsymbol{e}_K,h).
\end{eqnarray*}%
Let $\mu $ and $\nu $ be probability measures on $H$ and $K$ respectively.
Through the embedding, $\mu $ and $\nu $ can be viewed as probability
measures on $K\wr H.$ Consider the measure 
\begin{equation*}
q=\nu \ast \mu \ast \nu
\end{equation*}%
on $K\wr H$. This is the switch-walk-switch measure on $K\wr H$ with
switch-measure $\nu$ and walk-measure $\mu$.

Let $(X_{i})$ be the random walk on $H$ driven by $\mu ,$ and let $l(n,h)$
denote the number of visits to $h$ in the first $n$ steps:%
\begin{equation*}
l(n,h)=\#\{i:0\leq i\leq n,\text{ }X_{i}=h\}.
\end{equation*}%
Set also 
\begin{equation*}
l^{g}_*(n,h) =\left\{%
\begin{array}{ll}
l(n,h) & \mbox{ if } h\not\in \{ e_H,g\} \\ 
l(n,e_H)-1/2 & \mbox{ if } h=g \\ 
l(n,e_H)-1 & \mbox{ if } h=e_H.%
\end{array}
\right.
\end{equation*}

From \cite{Pittet2002}, probability that the random walk on $K\wr H$ driven
by $q$ is at $(h,g)\in K\wr H$ at time $n$ is given by 
\begin{equation*}
q^{(n)}((f,g))=\mathbf{E}\left( \prod\limits_{h\in H}\nu
^{(2l^g_*(n,h))}(f(h))\boldsymbol{1}_{\{X_{n}=g\} }\right)
\end{equation*}%
Note that $\mathbf{E}$ stands for expectation with respect to the random
walk $(X_i)_0^\infty$ on $H$ started at $e_H$.

From now on we assume that $\nu$ satisfies $\nu(e_K)=\epsilon >0$ so that 
\begin{equation*}
\epsilon \nu^{(n-1)}(e_K) \le \nu^{(n)}(e_K)\le \epsilon^{-1}
\nu^{(n-1))}(e_K).
\end{equation*}
Write $f\overset{C}{\asymp} g$ if $C^{-1}f\le g\le Cf$. Under these
circumstances, we have 
\begin{equation*}
q^{(n)}((\boldsymbol{e}_K,g)) \overset{1/\epsilon^{3}}{\asymp} \mathbf{E}%
\left( \prod\limits_{h\in H}\nu ^{(2l(n,h))}(e_K)\boldsymbol{1}_{\{X_{n}=g\}
}\right)
\end{equation*}
so that we can essentially ignore the difference between $l$ and $l_*$.

Set%
\begin{equation*}
F_{K}(n):=-\log \nu ^{(2n)}(e_K)
\end{equation*}%
so that, for any $g\in H$, 
\begin{equation}  \label{trick}
q^{(n)}((\boldsymbol{e}_K,g))\simeq \mathbf{E}\left( e^{
-\sum_{H}F_{K}(l(n,h))}\boldsymbol{1}_{\{X_{n}=g\}}\right).
\end{equation}

\begin{defin}[weak scaling assumption]
We say that $\nu $ satisfies the upper weak scaling assumption (US-$B_{n}$-$%
a_{n}$) if there exist a constant $c_0>0$ and a function $F:[0,\infty
)\rightarrow \lbrack 0,\infty )$ satisfying (S-$B_{n}$-$a_{n}$) and such
that 
\begin{equation}
\forall\, n\in \mathbb{N},\;\;c_{0}F(n)\leq F_{K}(n). 
\tag{US-$B_{n}$-$a_{n})$}
\end{equation}
We say that $\nu $ satisfies the lower weak scaling assumption \emph{(LS-$%
B_{n}$-$a_{n}$)} if there exist a constant $C_0<\infty $ and a function $%
F:[0,\infty )\rightarrow \lbrack 0,\infty )$ satisfying \emph{(-$B_{n}$-$%
a_{n}$ )} and such that 
\begin{equation}
\forall\, n\in \mathbb{N},\;\;F_K(n)\leq C_0F(n)  \tag{LS-$B_{n}$-$a_{n})$}
\end{equation}
If $F_{K}$ satisfies both the upper and lower conditions,%
\begin{equation}
\forall\,n\in \mathbb{N},\;\;c_{0}F(n)\leq F_{K}(n)\leq C_{0}F(n) 
\tag{WS-$B_{n}$-$a_{n}$}
\end{equation}%
then we say it satisfies the weak scaling assumption (WS-$B_{n}$-$a_{n}$).
\end{defin}

We can now use the large deviation asymptotics to estimate the 
return probability on wreath product $K\wr 
\mathbb{Z}
^{d}.$

\begin{theo}
\label{ReturnBounds}Let $\mu $ be a symmetric probability measure on $%
\mathbb{Z} ^{d}$ which satisfies the convergence assumption \emph{(C-$B_{n}$)%
}. Let $\nu $ be a symmetric probability measure on $K$ with $\nu(e_K)>0$.

\begin{itemize}
\item Assume that $\nu$ satisfies \emph{(US-$B_{n}$-$a_{n}$)}. Then the
switch-walk-switch measure $q=\nu \ast \mu \ast \nu $ on $K\wr \mathbb{Z}^d$
satisfies 
\begin{equation*}
\lim \sup_{n\rightarrow \infty }\frac{a_{n}}{n}\log q^{(n)}(e)\leq -k\left(
\eta ,c_{0}\widetilde{F}\right) .
\end{equation*}

\item Assume instead that $\nu $ satisfies \emph{(LS-$B_{n}$-$a_{n}$)}. Then
we have 
\begin{equation*}
\lim \inf_{n\rightarrow \infty }\frac{a_{2n}}{2n}\log q^{(2n)}(e)\geq
-k\left( \eta ,C_{0}\widetilde{F}\right) .
\end{equation*}
\end{itemize}
\end{theo}

\begin{rem}
Roughly speaking, this theorem says the following: Assume we know how to
normalize $\mu ^{(n)}$ on the base $\mathbb{Z}^d$ via a transformation $%
B_{n} $ so that it converges to a limiting distribution $\eta $. Assume we
know the behavior of the probability of return of the random walk on $K$
driven by $\nu $ in the sense that $\log (\nu^{(2n)}(e_K)) \simeq -F(n)$.
Then $q^{(n)}(e)\simeq \exp (-\frac{n}{a_{n}}),$ where $a_{n}$ can be
computed from the scaling relation%
\begin{equation*}
\frac{a_{n}\det (B_{a_{n}})}{n}F\left( \frac{n}{\det (B_{a_{n}})}\right)
\simeq 1.
\end{equation*}
\end{rem}

\begin{proof} The first statement follows immediately from (\ref{asy1}) 
in Theorem \ref{Asymptotic}. The second statement is deduced 
from (\ref{asy2}) as follows. Since $q$ is symmetric, we have  $q^{(2n)}(e_H)
\ge q^{(2n)}(g)$ for any $g\in K\wr \mathbb Z^d$.  In particular, if 
$B(r)=B_{\mathbb Z^d}(r)$
is the ball of radius $r$ in the lattice $\mathbb Z^d$ then, by (\ref{trick}),
\begin{eqnarray*}
\#B(r)q^{(2n)}(e) &\ge & c\sum_{h\in B(r)} q^{(2n)}(\boldsymbol e_K,h)\\
&\simeq & \mathbf E\left(e^{-\sum_H F_K(l(n,h))}\mathbf 1_{B(r)}(X_{2n})\right).
\end{eqnarray*}
Picking $r= Ra_{2n}$ and using the fact that $a_n$ has regular variation 
of order $\kappa>0$ (see Proposition \ref{Homogenous}), one easily deduces 
from (\ref{asy2}) that
$$\lim \inf_{n\rightarrow \infty }\frac{a_{2n}}{2n}\log q^{(2n)}(e)\geq -k\left(
\eta ,C_{0}\widetilde{F}\right) .
$$
as desired.
\end{proof}

\begin{exa}[$K\wr \mathbb{Z}^d$]
\label{LampsExponential} (See Example \ref{exa31}) Let $\mu$ be a symmetric
probability measure on $\mathbb{Z}^d$, $\mu \in \mbox{DNOA}_s(\eta)$ , $B_n=
\lfloor n^E\rfloor$, $\tra(E)=\tau$ (this implies $\tau\ge d/2$). Let $\nu$
be a symmetric probability measure on $K$ with $\nu(e_K)>0$ and such that 
\begin{equation*}
\log \nu^{(2n)}(e_K) \simeq - F(n).
\end{equation*}
Assume that $F$ is of the form $F(y)= y^\gamma \ell( y)$ where $\gamma\in
[0,1]$ and $\ell$ is a slow varying function (at infinity) such that $%
\ell(t^a\ell(t)^b) \sim c(a) \ell(t)$ for any $a>0$ and $b\in \mathbb{R}$.
(e.g., $\ell(t)= (\log t)^\beta $, $\beta\in \mathbb{R}$). Let $q$ be the
switch-walk-switch measure on $K\wr \mathbb{Z}^d$ associated with $\mu$ and $%
\nu$. Then 
\begin{equation*}
\log q^{(2n)}(e) \simeq -
\left(n^{\gamma+\tau(1-\gamma)}\ell(n)\right)^{1/(1+\tau(1-\gamma))}.
\end{equation*}
For a concrete example on $(\mathbb{Z}\wr \mathbb{Z})\wr \mathbb{Z}^d$, let $%
\mu$ be the uniform probability on 
\begin{equation*}
\{0,\pm s_1,\dots,\pm s_d\} \subset \mathbb{Z}^d
\end{equation*}
where $s_1,\dots,s_d$ are the unit vectors generating the square lattice $%
\mathbb{Z}^d$. Obviously, $\mu$ is in domain of normal attraction of the
Gaussian measure and $\tau= d/2$. On $K=\mathbb{Z}\wr \mathbb{Z}$, let $\nu$
be the switch-walk-switch measure on $\mathbb{Z}\wr \mathbb{Z}$ where both
the switch-measure and walk-measure are simple random walk on $\mathbb{Z}$
with holding. In this case, $F(y)= y^{1/3}(\log y)^{2/3}$ and $\gamma=1/3$
(see, e.g., \cite{Pittet2002}). Hence the measure $q=\nu*\mu*\nu$ on $K\wr 
\mathbb{Z}^d$ satisfies 
\begin{equation*}
\log q^{(2n)}(e )\simeq - n^{(1+d)/(3+d)}(\log n)^{2/(3+d)}.
\end{equation*}
We note that this result can also be obtained from Erschler's results \cite%
{Erschler2006}. Finally, keeping $K$ and $\nu$ as above, we replace the
measure $\mu$ on $\mathbb{Z}^d$ by the measure $\mu_\alpha(x)=c(1+\|x%
\|)^{-d-\alpha}$, $\alpha\in (0,2)$, $\|x\|=(\sum_1^d |x_i|^2)^{/2}$. Note
that $\mu_\alpha \in \mbox{DNOA}_s(\eta_\alpha)$ where $\eta_\alpha$ is the
rotationally symmetric $\alpha$-stable law on $\mathbb{R}^d$ and $\tau =
d/\alpha$. If we set $q_\alpha= \nu *\mu_\alpha *\nu$ then we obtain 
\begin{equation*}
\log q_\alpha ^{(2n)}(e) \simeq - n^{(1+2d/\alpha)/(3+2d/\alpha)}(\log
n)^{2/(3+2d/\alpha)}.
\end{equation*}
\end{exa}

The next theorem captures the fact that a better understanding of the return
probability on the lamp-group $K$ leads to a more precise asymptotics for $%
q^{(n)}(e).$

\begin{theo}
\label{ReturnAsymptotic} 
Let $\mu $ be a symmetric probability measure on $%
\mathbb{Z} ^{d}$ which satisfies the convergence assumption \emph{(C-$B_{n}$)%
}. Let $\nu $ be a symmetric probability measure on $K$ with $\nu(e_K)>0$.
Assume that the function $F_{K}(n)=-\log \nu ^{(2n)}(e_K)$ satisfies the
scaling assumption \emph{(S-$B_{n}$-$a_{n}$)}. Then the measure $q=\nu \ast
\mu \ast \nu $ on $K\wr \mathbb{Z}^d$ satisfies 
\begin{equation*}
\lim_{n\rightarrow \infty }\frac{a_{2n}}{2n}\log q^{(2n)}(e)=-k\left( \eta ,%
\widetilde{F}_K\right) .
\end{equation*}
\end{theo}

\begin{exa}
\label{LampPowerDecay} 
Referring to the setting of Theorem \ref%
{ReturnAsymptotic}, assume that $\nu ^{(2n)}(e_K)$ satisfies $\nu
^{(2n)}(e_K)\simeq n^{-\theta }$ so that $F_{K}(n)\sim \theta \log n$.
Assume $\mu $ is in the domain of normal attraction of $\eta $. Let $E\in %
\mbox{EXP}(\eta)$ such that $n^{-E}(\mu^{(n)})\Longrightarrow \eta$. Set $%
B_{n}=\left\lfloor n^{E}\right\rfloor $ (take integer values of all
entries). Let $\tau=\tra(E)$ be the trace of $E$. Solving for $t$ in the
scaling equation%
\begin{equation*}
\frac{t^{1+\tau}}{n}\log \left(\frac{n}{t^{\tau}}\right)=1,
\end{equation*}%
yields 
\begin{equation*}
a_n = \lfloor t \rfloor \sim \left( \frac{n}{\log n}\right) ^{\frac{1}{1+\tau%
}}.
\end{equation*}%
Then $F_{K}$ satisfies the scaling assumption 
\begin{equation*}
\lim_{n\rightarrow \infty }\frac{a_{n}\det (B_{a_{n}})}{n}F_{K}\left( \frac{n%
}{\det (B_{a_{n}})}y\right) =\theta \text{, for }y>0.
\end{equation*}%
Hence Theorem \ref{ReturnAsymptotic} yields 
\begin{equation*}
\lim_{n\rightarrow \infty }\frac{1}{(2n)^{\frac{\tau}{1+\tau}}\left( \log
2n\right) ^{\frac{1}{1+\tau}}}\log q^{(2n)}(e)=-k\left( \eta ,\widetilde{F}%
\right) ,
\end{equation*}%
where the limiting function $\widetilde{F}$ is given by $\widetilde{F}(y)
=\theta \cdot \mathbf{1}_{\{y>0\}}.$
\end{exa}

\subsection{Assorted examples}

In this section we describe a number of explicit applications of Theorems %
\ref{ReturnBounds} and \ref{ReturnAsymptotic}.

\begin{exa}
Let $\mathbb{Z}^d$ be equipped with the canonical generating $d$-tuple $%
S=(s_1,\dots,s_d)$ and fix $a=(\alpha_1,\dots,\alpha_d)\in (0,2)^d$.
Consider the probability measure $\mu_a$ given by 
\begin{equation}  \label{muaZd}
\mu_{a}(x)=\frac{1}{d}\sum_{i=1}^d\sum_{n\in\mathbb{Z}} \frac{c(\alpha_i)}{%
(1+|n|)^{1+\alpha_i}} \mathbf{1}_{\{n\}}(x_i),\;\;x=(x_1,\dots,x_d).
\end{equation}
This measure is quite obviously in the domain of normal operator attraction
of $\eta_a= \otimes_1^d \eta_{\alpha_i}$ where $\eta_{\alpha_i}$ is a
measure on $\mathbb{R}$ which is symmetric and $\alpha_i$-stable. In
particular, the diagonal $d\times d$ matrix $E_a$ with $i$-th diagonal entry 
$1/\alpha_i$ is in $\mbox{EXP}(\eta_a)$. The Dirichlet form $\mathcal{E}%
_{\eta_a}$ associated to the limit law $\eta_a$ is best described via
Fourier transform as $\mathcal{E}_{\eta_a}(f,f) =\sum_1^d c_i\int_{\mathbb{R}%
^d} |\hat{f}(\xi)|^2|\xi_i|^{2\alpha_i} d\xi$, $c_i>0$, $i=1,\dots d$ (the
scale parameters $c_i$ are related but not equal to $c(\alpha_i)$).

\begin{theo}
On $H=\mathbb{Z}^d$, consider the measures $\mu_a$ defined above, $a\in
(0,2)^d$. Define $\alpha\in (0,2)$ by 
\begin{equation*}
\frac{1}{\alpha}=\frac{1}{d}\sum_1^d\frac{1}{\alpha_i}.
\end{equation*}

\begin{enumerate}
\item Let $K$ be a finite group and let $\nu$ be the uniform measure on $K$.
On $K\wr \mathbb{Z}^d$, let $q_a=\nu*\mu_a*\nu$. Then there exists a
constant $k=k(d,a,|K|)$ such that 
\begin{equation*}
\log q_a^{(n)}(e) \sim - k n^{d/(d+\alpha)}.
\end{equation*}

\item Let $K=\mathbb{Z}^D$ and $\nu$ be a symmetric probability measure on $%
\mathbb{Z}^D$ with $\nu(e_K)>0$ which is in the domain of normal attraction
of an adapted strictly operator-stable law $\eta$. On $\mathbb{Z}^D\wr 
\mathbb{Z}^d$, let $q_a=\nu*\mu_a*\nu$. Then there exists a constant $%
k=k(d,a,D,\nu)$ such that 
\begin{equation*}
\log q_a^{(n)}(e) \sim - k n^{d/(d+\alpha)} (\log n)^{\alpha/(d+\alpha)}.
\end{equation*}
\end{enumerate}
\end{theo}
\end{exa}

\begin{exa}
Set $H=\mathbb{Z}^d$, $K=\mathbb{Z}^D$, $G=K\wr H= \mathbb{Z}^D\wr \mathbb{Z}%
^d$. One natural set of generators of $G$ is obtained by joining the
canonical generators of $H=\mathbb{Z}^d$ and $K=\mathbb{Z}^D$ as follows.
Let $(s^H_i)_1^d$ and $(s^K_i)_1^D$ be the canonical generators of $H$ and $%
K $, respectively. Let $S=(s_i)_1^{d+D}$ be the generating tuple of $G$
given by 
\begin{equation*}
s_i=(\boldsymbol{e}_K,s^H_i) \mbox{ for }i\in \{1,\dots,d\} \mbox{ and }
s_{i}=(\boldsymbol{s}^K_i,e_H) \mbox{ for } i\in \{d+1,\dots,d+D\}.
\end{equation*}
Of course, $e_K=0$ in $\mathbb{Z}^D$ and $e_H=0$ in $\mathbb{Z}^d$. Let 
\begin{equation*}
a=(\alpha_1,\dots,\alpha_{d+D})\in (0,2)^{d+D}
\end{equation*}
be a $(d+D)$-tuple. Let $b=b(a)=(\beta_i)_1^d$ and $c=c(a)=(\gamma_i)_1^D$
with $\beta_i=\alpha_i$, $i=1,\dots,d$ and $\gamma_i=\alpha_{d+i}$, $%
i=1,\dots,D$. Let $\mu^H_b,\mu^K_c$ be the probability measures on $H=%
\mathbb{Z}^d,K=\mathbb{Z}^D$, respectively, defined at (\ref{muaZd}). Let $q$
be the switch-walk-switch measure on $G=K\wr H$ given by $%
q=\mu^K_c*\mu^H_b*\mu^K_c$. The theorem stated above applies and yields 
\begin{equation*}
\log q^{(n)}(e)\sim -k(d,D,a) n^{d/(d+\beta)} (\log
n)^{\beta/(d+\beta)},\;\; \frac{1}{\beta}=\frac{1}{d}\sum_1^d\frac{1}{\beta_i%
}.
\end{equation*}

For $S$ and $a$ as defined above, let $\mu_{S,a}$ the the probability
measure on $G=K\wr H$ defined by
\begin{equation}  \label{muSa}
\mu_{S,a}(g)=\frac{1}{k}\sum_i \sum _{n\in \mathbb{Z}}\mathbf{1}%
_{\{s_i^n\}}(g) \mu_i(n), \;\;\;\mu_i(n)=c_i(1+|n|)^{-1-\alpha_i}.
\end{equation}
In words, this walk takes steps along the (discrete) one parameter groups $%
\langle s_i \rangle =\{s_i^{n}, n\in \mathbb{Z}\}\subset G$ and the steps
along $\langle s_i\rangle$ are distributed according to a symmetric
stable-like power law with exponent $\alpha_i$. These measures $\mu_{S,a}$
are very natural from an algebraic point of view and one expects that the
properties of the associated random walks depend in interesting way on the
structure of the group $G$, the generating $k$-tuple $S$ and the choice of
the $k$-dimensional parameter $a$.

The Dirichlet forms $\mathcal{%
E}_{\mu_{S,a}}$ and $\mathcal{E}_q$ associated with the measures $\mu_{S,a}$
and $q$ on $G$ satisfy 
\begin{equation*}
\mathcal{E}_{\mu_{S,a}} \simeq \mathcal{E}_q.
\end{equation*}
Hence it follows from \cite{Pittet2000} that 
\begin{equation*}
\log \mu_{S,a}^{(n)}(e) \simeq - n^{d/(d+\beta)} (\log n)^{\beta/(d+\beta)}
\end{equation*}
where $\beta$ is as above. Note that $\beta$ depends only on the first $d$
coordinates of the parameter $a=(\alpha_i)_1^{d+D}$. In this sense, the
random walks associated with the collection of the measures $\mu_{S,a}$ when 
$a$ varies can distinguish among the $d+D$ generators $s_i$, $1\le i\le d+D$
of $K\wr H$ between those which come from $H$ and those which come from $K$.
\end{exa}

\begin{exa}
\label{Iterative} Consider the iterated wreath product 
\begin{equation*}
(\dots ( \mathbb{Z}_{2}\wr \mathbb{Z} ^{d_{1}})\wr \mathbb{Z} ^{d_{2}})\wr
\dots)\wr \mathbb{Z} ^{d_{k}}.
\end{equation*}%
Note that $\wr $ is not associative so that this iterated wreath product is
different from the iterated wreath product $%
\mathbb{Z}
_{2}\wr (\dots\wr (%
\mathbb{Z}
\wr 
\mathbb{Z}
)\dots)$ considered in \cite{Erschler2006}. Here we are iterating the lamps
while in \cite{Erschler2006} the base is iterated. Set 
\begin{equation*}
\gamma_i=\sum_1^i\frac{d_{j}}{\alpha_{j}},\;i=1,\dots,k.
\end{equation*}

For each $i=1,\dots,k$, fix $\alpha_i\in (0,2)$ and a probability measure $%
\mu_i$ on $\mathbb{Z}^{d_i}$ which is symmetric, satisfies $\mu_i(0)>0$ and
is in the domain of normal attraction of the rotationally $\alpha_i$-stable
law $\eta_i$ on $\mathbb{R}^{d_i}$. Let $q_0$ be the uniform measure on $%
\mathbb{Z}_2=\{0,1\}$. Iteratively, define the switch-walk-switch
probability measure 
\begin{equation*}
q_i= q_{i-1}*\mu_i*q_{i-1}
\end{equation*}
on $(\dots( \mathbb{Z}_{2}\wr \mathbb{Z} ^{d_{1}})\wr \mathbb{Z}
^{d_{2}})\wr \dots)\wr \mathbb{Z} ^{d_{i}}$.

Applying Corollary \ref{ReturnAsymptotic} iteratively, we obtain 
\begin{equation*}
\lim_{n\rightarrow \infty }n^{-\frac{\gamma _{k}}{1+\gamma _{k}}} \log
q_k^{(n)}(e)=-c_{k}
\end{equation*}%
where the constant $c_{k}$ can be obtained as follows. The constant $c_{1}$
is given by \cite{Donsker1979} whereas, for $2\leq i\leq k$ and referring to
(\ref{asy1})-(\ref{asy2}), $c_{i}=k\left( \upsilon _{i},\widetilde{F}%
_{i}\right) $ where 
\begin{equation*}
\widetilde{F}_{i}(y)=c_{i-1}y^{\frac{\gamma _{i-1}}{1+\gamma _{i-1}}}.
\end{equation*}
Similarly, we can consider the iterated wreath product 
\begin{equation*}
(\dots( \mathbb{Z} ^{d_{0}}\wr \mathbb{Z} ^{d_{1}})\wr \mathbb{Z}
^{d_{2}})\wr ...)\wr \mathbb{Z} ^{d_{k}},
\end{equation*}
starting with lamp group $\mathbb{Z}^{d_{0}}$ instead of $\mathbb{Z}_2$ and $%
q_0=\mu_0$ in the domain of normal attraction of the rotationally symmetric $%
\alpha _{0}$-stable distribution on $\mathbb{R }^{d_{0}}.$ In this case, we
obtain 
\begin{equation*}
\lim_{n\rightarrow \infty } [n^{\gamma_k/(1+\gamma_k)} (\log
n)^{1/(1+\gamma_k)}]^{-1} \log q^{(n)}(e)=-c_{k}.\;\;
\end{equation*}%
The constant $c_{k}$ can be obtained iteratively with $c_{1}=k\left( \eta ,%
\frac{d_{0}}{\alpha _{0}}\mathbf{1}_{\{y>0\}}\right)$ and $c_{i}=k\left(
\upsilon _{i},\widetilde{F}_{i}\right) , $ with $\widetilde{F}_{i}$ as above
for $1<i\le k$.
\end{exa}

\subsection{Application to fastest decay under moment conditions}

This section describes applications of Theorem \ref{ReturnBounds} to the
computation of the group invariants $\widetilde{\Phi}_{G,\rho}$ introduced
in \cite{Bendikov}. Recall that \cite{Pittet2000} introduce a group
invariant $\Phi_G$ which is a decreasing function of $n$ (defined up to the
equivalence relation $\simeq$) such that 
\begin{equation*}
\phi^{(2n)}(e)\simeq \Phi_G(n)
\end{equation*}
for all finitely supported symmetric probability measure $\phi$ with
generating support.

Let $\rho $ be a function 
\begin{equation*}
\rho :G\rightarrow \lbrack 1,\infty ).
\end{equation*}%
The weak $\rho $-moment of the probability measure $\mu $ is defined as 
\begin{equation*}
W(\rho ,\mu ):=\sup_{s>0}s\mu (x:\rho (x)>s).
\end{equation*}

\begin{defin}[Definition 2.1 \protect\cite{Bendikov}: Fastest decay under
weak $\protect\rho $-moment]
Let $G$ be a locally compact unimodular group. Fix a compact symmetric
neighborhood $\Omega $ of $e$. Let $\widetilde{\mathcal{S}}_{G,\rho
}^{\Omega ,K}$ be the set of all symmetric continuous probability densities $%
\phi $ on $G$ with the properties that $\left\Vert \phi \right\Vert _{\infty
}\leq K$ and $W(\rho ,\phi d\lambda )\leq K\sup_{\Omega ^{2}}\{\rho \}.$ Set%
\begin{equation*}
\widetilde{\Phi }_{G,\rho }^{\Omega ,K}(n):=\inf \{\phi ^{(2n)}(e):\phi \in 
\widetilde{\mathcal{S}}_{G,\rho }^{\Omega ,K}\}.
\end{equation*}
\end{defin}

Here we will only consider the case when $G$ is finitely generated and $\rho$
is one of the power function $\rho _{\alpha }(x)=(1+|x|)^\alpha$, $\alpha\in
(0,2)$ where $\left\vert \cdot \right\vert $ is the word distance on a fixed
Cayley graph of $G.$ We are concerned with the decay of $\widetilde{\Phi }%
_{G,\rho_\alpha }^{\Omega ,K}$ when $n$ is large. By Proposition 1.2 \cite%
{Bendikov}, we can drop the reference to $\Omega $ and $K$. Lower bounds on $%
\widetilde{\Phi }_{G,\rho }$ follow from general comparison and
subordination results, see \cite{Bendikov}. Here, we are interested in
obtaining upper bounds on $\widetilde{\Phi}_{G,\rho}$.

By definition, for any probability measure $\phi $ on $G$ which satisfies
the weak $\rho $-moment condition, $n\mapsto \phi ^{(2n)}(e)$ provides an
upper bound for $\widetilde{\Phi }_{G,\rho }$. When $G$ is a wreath product $%
G=K\wr 
\mathbb{Z}
^{d},$ we can use measures of the form $\phi =\nu \ast \mu \ast \nu $ and
apply Theorem \ref{ReturnBounds} to estimate $\phi ^{(2n)}(e).$ Also because
of the natural embedding of $K$ and $%
\mathbb{Z}
^{d}$ in the wreath product $K\wr 
\mathbb{Z}
^{d},$ it's not hard to estimate the needed weak $\rho$-moment of $\phi$. We
shall see that, in certain cases, the measures $\phi $ of this type actually
achieve the fastest decay rate given by $\widetilde{\Phi}_{G,\rho}$, up to
the equivalence relation $\simeq$. This technique was already used in \cite[%
Theorem 5.1]{Bendikov} to determine $\widetilde{\Phi }_{%
\mathbb{Z}
_{2}\wr 
\mathbb{Z}
^{d},\rho _{\alpha }}$. In this case, the classical result of Donsker and
Varadhan \cite{Donsker1979} is all one needs. In the examples below, we use
Theorem \ref{ReturnBounds} to obtain precise upper bounds on $\Phi_{K\wr 
\mathbb{Z}^d}$ in some other cases.

\begin{exa}
In this example we consider $G=K\wr \mathbb{Z}^d$ when $K$ is either finite
or has polynomial growth or has exponential volume growth and $%
\Phi_K(n)\simeq \exp(-n^{1/3})$. The first case is already treated in \cite%
{Bendikov}. We note that these three cases exhaust all possibilities when $K$
is a polycyclic group. The third case also covers the situations when $K$ is
the Baumslag-Solitar group or the lamplighter group $\mathbb{Z}_2\wr \mathbb{%
Z}$.

\begin{theo}
Fix $\alpha\in (0,2)$. Let $G$ be the group $K\wr \mathbb{Z}^d$.

\begin{enumerate}
\item Assume that $K$ is finite. Then 
\begin{equation*}
\log \widetilde{\Phi}_{G,\rho_\alpha}(n) \simeq -n^{d/(d+\alpha)}.
\end{equation*}

\item Assume that $K$ has polynomial volume growth. Then 
\begin{equation*}
\log \widetilde{\Phi}_{G,\rho_\alpha}(n) \simeq -n^{d/(d+\alpha)} (\log
n)^{\alpha/(d+\alpha)}.
\end{equation*}

\item Assume that $K$ has exponential growth and satisfies $\Phi_K(n)\simeq
\exp(-n^{1/3})$. Then 
\begin{equation*}
\log \widetilde{\Phi}_{G,\rho_\alpha}(n) \simeq -n^{(d+1)/(d+1+\alpha)}.
\end{equation*}
\end{enumerate}
\end{theo}

\begin{proof}
The lower bounds can be obtained by applying  \cite[Theorem 3.3]{Bendikov}. 
For this purpose, one needs to compute the function $\Phi_G$.
For $q=\nu \ast
\mu \ast \nu $ on $G=K\wr 
\mathbb{Z}
^{d},$ where $\mu $ and $\nu $ are associated with simple random walk (with holding) on $%
\mathbb{Z}
^{d}$ and $K$ respectively, we can apply Theorem \ref{ReturnBounds} to
obtain  $q^{(2n)}(e) \simeq \Phi_{G}(n).$ 
The case when $K$ is finite is already treated in \cite{Pittet2002,Bendikov}.
When $K$ has polynomial
volume growth, then $\nu ^{(2n)}(o)\asymp n^{-\frac{D}{2}}$ and, as in Example %
\ref{LampPowerDecay},%
\begin{equation*}
\lim_{n\rightarrow \infty }\frac{1}{n^{\frac{d}{d+2}}(\log n)^{\frac{2}{d+2}}%
}\log q^{(2n)}(e)=-c_{q}.
\end{equation*}%
If  $K$ is such that $\Phi_K(n)\simeq \exp(-n^{1/3})$ then  Example 
\ref{LampsExponential} yields
\begin{equation*}
\log
q^{(2n)}(e)\simeq \log \Phi_G(n)\simeq 
-n^{\frac{d+1}{d+3}}.
\end{equation*}%
These estimates on $\Phi_{G}$ allow us to appeal to
\cite[Theorem 3.3]{Bendikov} to obtain the stated lower
bounds for $\Phi_{G,\rho_\alpha}$.

To prove the  stated upper bounds, it suffices to exhibit a
probability measure measure in
$\widetilde{\mathcal{S}}_{G,\rho _\alpha}$ that has the proper decay. 
On $\mathbb{Z}
^{d}$, set 
\begin{equation*}
\mu _{\alpha }(x)=\frac{c_{\alpha }}{(1+\| x\| )^{\alpha +1}%
}.
\end{equation*}%
Then $\mu_\alpha$ is in the domain of normal attraction of 
the rotationally symmetric $\alpha $%
-stable distribution on $%
\mathbb{R}
^{d}$ and it has a finite  weak $\alpha $-moment.

In the  case when $K$ is of polynomial volume growth, take $q_{\alpha }=\nu
\ast \mu _{\alpha }\ast \nu $, where $\nu $ is simple random walk on $K$.
Then $\nu \ast \mu _{\alpha }\ast \nu $ has weak $\alpha$-moment and, by
Example \ref{LampPowerDecay},%
\begin{equation*}
\lim_{n\rightarrow \infty }\frac{1}{n^{\frac{d}{d+\alpha }}(\log n)^{\frac{%
\alpha }{d+\alpha }}}\log q_{\alpha }^{(2n)}(e)=-c_{q_{\alpha }}.
\end{equation*}%
Therefore in this case%
\begin{equation*}
\log \widetilde{\Phi }_{G,\rho _{\alpha }}(n)\leq -cn^{\frac{d}{d+\alpha }%
}(\log n)^{\frac{\alpha }{d+\alpha }}.
\end{equation*}
This matches the previously proved lower bound.

In the second case, when $K$ has
exponential growth, let $U$ be a symmetric generating set of $K.$ As in
\cite[Theorem 4.10]{Bendikov}, pick $p_{i}=c_\alpha 4^{-i\alpha }$ 
with $\sum_1^\infty p_i=1$ and set 
\begin{equation*}
\nu _{\alpha }=\sum_{i=1}^{\infty }\frac{p_i}{\left\vert
U^{4^{i}}\right\vert }\mathbf{1}_{U^{4^{i}}}.
\end{equation*}%
Then $\nu _{\alpha }$ has weak $\alpha $-moment on $K$ and, 
by \cite[Theorem 4.1]{Bendikov} 
$$\nu _{\alpha }^{(n)}(e_K)\leq \exp (-cn^{\frac{1}{1+\alpha }}).$$
Then $\nu _{\alpha }\ast \mu _{\alpha }\ast \nu _{\alpha }$ has weak $\alpha 
$-moment on $G$. Applying Theorem \ref{ReturnBounds}
and the computations of Example \ref{LampsExponential} to $%
q_{\alpha }=\nu _{\alpha }\ast \mu _{\alpha }\ast \nu _{\alpha },$ we obtain
\begin{equation*}
\lim_{n\rightarrow \infty }\sup \frac{1}{n^{\frac{d+1 }{d+\alpha +1}}}%
\log q_{\alpha }^{(2n)}(e)\leq -c_{q_{\alpha }}.
\end{equation*}
This gives the desired upper bound on $\Phi_{G,\rho_\alpha}$.
\end{proof}
\end{exa}

\begin{exa}
Consider the iterated wreath product 
\begin{equation*}
G=(\dots( K\wr \mathbb{Z} ^{d_{1}})\wr \dots)\wr \mathbb{Z}
^{d_{r}},\;\;d_{i}\in \mathbb{N} _{+}.
\end{equation*}
Set 
\begin{equation*}
d=\sum_1^rd_i.
\end{equation*}
Fix $\alpha\in (0,2)$. If $K$ is finite then we we have 
\begin{equation*}
\log \widetilde{\Phi }_{G,\rho _{\alpha }}(n)\simeq -n^{\frac{d }{\alpha +d}%
}.
\end{equation*}
If $K$ has polynomial volume growth, then 
\begin{equation*}
\log \widetilde{\Phi }_{G,\rho _{\alpha }}(n)\simeq -cn^{\frac{d}{\alpha +d}%
}(\log n)^{\frac{\alpha }{\alpha +d}}.
\end{equation*}
These are the results stated as Theorem \ref{theo-alphaiter} in the
introduction.

\begin{proof} 
As in the previous example, the lower bounds follows from 
\cite[Theorem 3.3]{Bendikov} and a lower bound on $\log \Phi_G$.
By example \ref{Iterative},      
\begin{equation*}
\log \Phi_G(n)\simeq n^{\frac{d}{%
2+d}}(\log n)^{\frac{2}{2+d}}.
\end{equation*}%
Hence \cite[Theorem 3.3]{Bendikov} gives%
\begin{equation*}
\log \widetilde{\Phi }_{G,\rho _{\alpha }}(n)\geq -C_{\alpha }n^{\frac{%
d}{\alpha +d}}(\log n)^{\frac{\alpha }{\alpha
+d}}.
\end{equation*}

For the upper bound, let $\mu _{\alpha ,i}$ be a symmetric $\alpha $-stable
like probability measure on $%
\mathbb{Z}
^{d_{i}}$. Let $q_{\alpha ,1}=\mu _{\alpha ,0}\ast \mu _{\alpha ,1}\ast \mu
_{\alpha ,0},$ and iteratively define $q_{\alpha ,i+1}=q_{\alpha ,i}\ast \mu
_{\alpha ,i+1}\ast q_{\alpha ,i}.$ Then it's clear that $q_{\alpha ,r}$ has
a finite weak $\alpha $-moment on $G$ and, as in example \ref{Iterative}, 
\begin{equation*}
\lim_{n\rightarrow \infty }\frac{1}{n^{\frac{d}{\alpha
+d}}(\log n)^{\frac{\alpha }{\alpha +d}}}\log
q_{\alpha ,r}^{(2n)}(e)=-c_{\alpha ,r}.
\end{equation*}%
Therefore%
\begin{equation*}
\log \widetilde{\Phi }_{G,\rho _{\alpha }}(n)\leq -cn^{\frac{d
}{\alpha +d}}(\log n)^{\frac{\alpha }{\alpha +d}
}.
\end{equation*}
\end{proof}
\end{exa}

\section{Donsker and Varadhan type large deviations}

\label{sec-DV}

\setcounter{equation}{0}

The goal of this section is to outline the proof of Theorem \ref{Asymptotic}%
, the key result of this article. The proof follows \cite{Donsker1979}
closely. Several other classical sources are also needed to put together the
necessary details.

\subsection{Statement of the large deviation principle in $L^1$}

On $\mathbb{Z}^d$, we fix a symmetric probability measure $\mu$ and an
operator-stable law $\eta$ such that the convergence assumption (C-$B_n$) is
satisfied.

We need to introduce some notation from \cite{Donsker1979} in order to state
the results. Let $\pi $ be the projection map $\pi :%
\mathbb{R}
^{d}\rightarrow 
\mathbb{R}
^{d}/%
\mathbb{Z}
^{d}$, and let $\mathbb{T}$ denote the $d$-dimensional torus which we also
identify with the fundamental domain $[-\frac{1}{2},\frac{1}{2})^{d}.$

For $\lambda >0,$ set 
\begin{equation*}
\mathcal{L}_{\lambda }^{(n)}=\pi \left( B_{\left\lfloor \lambda
a_{n}\right\rfloor }^{-1}(%
\mathbb{Z}
^{d})\right) .
\end{equation*}%
That is, we take the image of the original lattice $%
\mathbb{Z}
^{d}$ under the transformation $B_{\left\lfloor \lambda a_{n}\right\rfloor
}^{-1} $, and project it to the torus $\mathbb{T}.$ Then $\mathcal{L}%
_{\lambda }^{(n)}$ is a cocompact lattice on $\mathbb{T}$ and the volume of
the fundamental domain $\mathbb{T}/\mathcal{L}_{\lambda }^{(n)}$ is $%
\left\vert \det B_{\left\lfloor \lambda a_{n}\right\rfloor }^{-1}\right\vert
.$ This is the case because we assume that the matrices $B_m$, $m=1,2\dots$,
have integer entries so that $B_m \mathbb{Z}^d\subset \mathbb{Z}^d$.

In what follows, symbols decorated with $\widetilde{}$ \ are always used to
describe quantities associated with the projected random walk on the torus.
Note that the construction depends on the choice of sequence $a_{n}$ and
parameter $\lambda ,$ for simplicity we will drop reference to $a_{n}$ and $%
\lambda $ when no confusion arises.

Under the projection map $\pi ,$ we can push forward the measure $%
B_{\left\lfloor \lambda a_{n}\right\rfloor }^{-1}\mu $ on $B_{\left\lfloor
\lambda a_{n}\right\rfloor }^{-1}(%
\mathbb{Z}
^{d})$ to a measure $\widetilde{\mu }_{n,\lambda }$ on $\mathcal{L}_{\lambda
}^{(n)}$, that is

\begin{equation*}
\widetilde{\mu }_{n,\lambda }(y)=\sum_{x\in 
\mathbb{Z}
^{d}:\pi \left( B_{\left\lfloor \lambda a_{n}\right\rfloor }^{-1}(x)\right)
=y}\mu (x).
\end{equation*}

Let $\widetilde{S}_{k}^{(n)}$ be the random walk on $\mathcal{L}_{\lambda
}^{(n)} $ associated with $\widetilde{\mu }_{n,\lambda }$, starting at $0.$
It's easy to check that%
\begin{equation*}
\widetilde{S}_{k}^{(n)}\overset{\mbox{\tiny law}}{=}\pi \left(
B_{\left\lfloor \lambda a_{n}\right\rfloor }^{-1}(S_{k})\right) .
\end{equation*}

Consider the occupation time measure $\widetilde{L}_{k}^{(n)}$ defined as 
\begin{equation*}
\widetilde{L}_{k}^{(n)}(A)=\frac{1}{k}\sum_{j=1}^{k}\chi _{A}\left( 
\widetilde{S}_{k}^{(n)}\right) ,
\end{equation*}%
for any Borel set $A$ in $\mathbb{T}$.

For $\mathbf{T}=\mathbb{T}$ or $\mathbf{T}=\overline{\Omega}$ with $\Omega$
an open set in $\mathbb{R}^d$, let $\mathcal{M}_{1}(\mathbf{T})$ be the
space of probability measures on $\mathbf{T}$ endowed with the weak
topology. Let $\mathcal{L}_1(\mathbf{T})$ the space of all probability
densities on $\mathbf{T}$ endowed with the $L^1$-topology.

Let $P_{k}^{(n)}$ be the distribution of $\widetilde{L}_{k}^{(n)}$ in $%
\mathcal{M}_{1}(\mathbb{T})$, a measure on measures. Define the scaled
indicator function $\chi _{n}:[-\frac{1}{2},\frac{1}{2})^{d}\rightarrow 
\mathbb{R}$ by setting 
\begin{equation*}
\chi _{n}(x)=|\det B_{\left\lfloor \lambda a_{n}\right\rfloor }|^{-1}\chi
_{B_{\left\lfloor \lambda a_{n}\right\rfloor }^{-1}\left( [-\frac{1}{2},%
\frac{1}{2})^{d}\right) }.
\end{equation*}%
Define%
\begin{equation*}
\widetilde{L}_{k}^{n}=\widetilde{L}_{k}^{(n)}\ast \chi _{n}.
\end{equation*}%
Let $\widetilde{P}_{k}^{n}$ be the distribution of $\widetilde{L}_{k}^{n}$
in $\mathcal{M}_{1}(\mathbb{T})$. With this mollification, $\widetilde{L}%
_{k}^{n}$ is absolutely continuous with respect to Lebesgue measure on $%
\mathbb{T}$. Let $\widetilde{f}_{k,\lambda }^{(n)}$ denote the density of $%
\widetilde{P}_{k}^{n}$ with respect to Lebesgue measure. Let $Q_{k,\lambda
}^{(n)}$ be the distribution of $\widetilde{f}_{k,\lambda }^{(n)}$ in $%
\mathcal{L}_{1}(\mathbb{T}).$

We will use the following function spaces (this notation is consistent with 
\cite{Donsker1975,Donsker1979}): 
\begin{eqnarray*}
\mathcal{U} &=&\{u\in C^{\infty }(\mathbb{%
\mathbb{R}
}^{d}),\inf u>0,\sup u<\infty \}, \\
\mathcal{U}_{\mathbb{T}} &=&\{u\in C^{\infty }(\mathbb{T}),u>0\}, \\
\mathcal{F}_{\mathbb{T}} &=&\{f\in C^{\infty }(\mathbb{T}),f\geq
0,\left\Vert f\right\Vert _{1}=1\} , \\
\mathcal{F}_{\Omega } &=&\{f\in C_c^{\infty }(\Omega),f\geq 0,\left\Vert
f\right\Vert _{1}=1\},
\end{eqnarray*}
where $\Omega$ is an open subset of $\mathbb{R}^d$.

\begin{theo}[Large deviation principle in $\mathcal{L}_{1}(\mathbb{T})$]
\label{TorusLDP} Assume that the convergence assumption \emph{(\ref%
{ConvAssum})} is satisfied. Let $a_{n}$ to be any sequence of positive
integers increasing to infinity and satisfying $a_{n}\left\vert \det
B_{a_{n}}\right\vert \leq n$. 
Let $Q_{n,\lambda }^{(n)}$ be the distribution
of $\widetilde{f}_{n,\lambda }^{(n)}$ on $\mathcal{L}_{1}(\mathbb{T}).$ Then
we have the large deviation principle in the strong $\mathcal{L}_1(\mathbb{T}%
)$ topology. Namely, for any Borel set $D$ in $\mathcal{L}_1(\mathbb{T}),$ 
\begin{eqnarray*}
-\lambda ^{-1}\inf_{f\in D^{\circ }}I_{L_{\widetilde{\eta }}}(f) &\leq
&\liminf_{n\rightarrow \infty }\frac{1}{n/{a_{n}}}\log Q_{n,\lambda
}^{(n)}(D) \\
&\leq &\limsup_{n\rightarrow \infty }\frac{1}{n/{a_{n}}}\log Q_{n,\lambda
}^{(n)}(D)\leq -\lambda ^{-1}\inf_{f\in \overline{D}}I_{L_{\widetilde{\eta }%
}}(f),
\end{eqnarray*}%
and the rate function is given by 
\begin{equation*}
I_{L_{\widetilde{\eta }}}(f)=-\inf_{u\in \mathcal{U}_{\mathbb{T}}}\int_{%
\mathbb{T}}\frac{L_{\widetilde{\eta }}u}{u}(x)f(x)dx=\mathcal{E}_{\widetilde{%
\eta }}(\sqrt{f},\sqrt{f}).
\end{equation*}
\end{theo}

This result will be useful in the upper bound direction. To obtain a lower
bound, we need to have a version with Dirichlet boundary condition.

Let $L_{k}^{(n)}$ be the occupation time measure of the random walk $%
S_{k}^{(n)}=B_{a_{n}}^{-1}(S_{k}).$ Perform the same mollification as above
but on $%
\mathbb{R}
^{d}$, setting 
\begin{equation*}
L_{k}^{n}=L_{k}^{(n)}\ast \chi _{n}.
\end{equation*}%
Then $L_{k}^{n}$ is absolutely continuous with respect to Lebesgue measure.
Let $f_{k}^{(n)}$ denotes the corresponding density. Let $\mathcal{G}$ be
the collection of all bounded domain $\Omega $ in $%
\mathbb{R}
^{d}$ such that $0\in \Omega $ and $\partial \Omega $ has Lebesgue measure $%
0.$ For any Borel set $A\subset \mathcal{L}_1(\overline{\Omega} ),$ define%
\begin{equation*}
Q_{k,\Omega }^{(n)}(A):=P\left( f_{k}^{(n)}\in A\right) .
\end{equation*}%
That is, $Q_{k,\Omega }^{(n)}$ is the distribution of the occupation time
measure of $S_{j}^{(n)}$ at time $k$ with Dirichlet boundary on $\partial
\Omega .$ As in the case of the projected version, we have the following
large deviation principle.

\begin{theo}[Large deviation principle in $\mathcal{L}_1(\overline{\Omega} )$%
]
\label{DirichletLDP}Under the convergence assumption \emph{(C-$B_{n}$)}, let 
$a_{n}$ be any sequence of positive integers increasing to infinity
satisfying $a_{n}\left\vert \det B_{a_{n}}\right\vert \leq n.$ Let $%
Q_{n,\Omega }^{(n)}$ be the distribution of $f_{n}^{(n)}$ in $\mathcal{L}_1(%
\overline{\Omega} )$. Then we have large deviation principle in the strong $%
\mathcal{L}_1(\overline{\Omega} )$ topology. Namely, for any Borel set $%
A\subset \mathcal{L}_1(\overline{\Omega} ),$ 
\begin{eqnarray*}
-\inf_{f\in A^{\circ }}I_{L_{\eta }}(f) &\leq &\liminf_{n\rightarrow \infty }%
\frac{1}{n/{a_{n}}}\log Q_{n,\Omega }^{(n)}(A) \\
&\leq &\limsup_{n\rightarrow \infty }\frac{1}{n/{a_{n}}}\log Q_{n,\Omega
}^{(n)}(A)\leq -\inf_{f\in \overline{A}}I_{L_{\eta }}(\eta ),
\end{eqnarray*}%
and 
\begin{equation*}
I_{L_{\eta }}(f)=-\inf_{u\in \mathcal{U}}\int_{\Omega }\frac{L_{\eta }u}{u}%
(x)f(x)dx=\mathcal{E}_{\eta }(\sqrt{f},\sqrt{f}).
\end{equation*}
\end{theo}

The outline of the proof of these results is given in Section \ref{LDPproof}.
It follows \cite{Donsker1979} closely.

\subsection{Asymptotics of functional expressions}

\label{sec-DVBK} Throughout this short section, we fix a symmetric
probability measure $\mu $ on $\mathbb{Z}^{d}$ and an operator-stable law $%
\eta $ on $\mathbb{R}^{d}$ such that the convergence assumption (C-$B_{n}$)
and scaling assumption (S-$B_{n}$-$a_{n}$) of Definitions \ref{def-conv}-\ref%
{def-scal} are satisfied. In particular, in what follows, $(a_{n})$ is the
non-decreasing and regularly varying sequence of integers provided by
Definition \ref{def-scal} (see also Proposition \ref{Homogenous}). The
functions $F$ and $\widetilde{F}$ are as in Definition \ref{def-scal}. Let $%
(l(n,x))_{x\in \mathbb{Z}^{d}}$ be the occupation time vector up to time $n$
for the random walk driven $\mu $. The goal of this subsection is to use the
large deviation principles in $L^1$ to prove Theorem \ref{Asymptotic}.

\begin{pro}
\label{LowerBound} Under the above hypotheses, we have the lower bound%
\begin{eqnarray*}
\lefteqn{\lim \inf_{n\rightarrow \infty }\frac{a_{n}}{n}\log E\left( \exp
\left( -\sum_{x\in\mathbb{Z}^{d}}F(l(n,x))\right)\mathbf 1_{\{ \mbox{\em
\scriptsize supp}(L_n^{(n)})\subset\Omega\}}\right)} \hspace{1in} && \\
&\geq &-\inf_{f\in \mathcal{F}_{\Omega }}\left\{ \mathcal{E}_{\eta }(\sqrt{f}%
,\sqrt{f})+\int_{\Omega }\widetilde{F}(f(x))dx\right\} .
\end{eqnarray*}
\end{pro}

\begin{proof} The proof is essentially the same as for 
\cite[Lemma 4.2]{Biskup2001}. Use the lower bound
in Theorem \ref{DirichletLDP} and Varadhan's lemma 
(see \cite[Theorems 2.2, 2.3]{Vb}).
\end{proof}

\begin{pro}
\label{UpperBound} Under the above hypotheses, we have upper bound%
\begin{eqnarray*}
&&\lim \sup_{n\rightarrow \infty }\frac{a_{n}}{n}\log E\left( \exp \left(
-\sum_{x\in\mathbb{Z}^{d}}F(l(n,x))\right) \right) \\
&\leq &-\sup_{\lambda >0}\inf_{f\in \mathcal{F}_{\mathbb{T}}}\left\{ \lambda
^{-1}\mathcal{E}_{\widetilde{\eta }}(\sqrt{f},\sqrt{f})+c_{0}\lambda
^{(1-\gamma )\tra E}\int_{\mathbb{T}}\widetilde{F}(f(x))dx\right\} .
\end{eqnarray*}
\end{pro}

\begin{proof}
First, since $F$ is sub-additive, write
\begin{eqnarray*}
\lefteqn{E\left( \exp \left( -\sum_{x\in 
\mathbb{Z}
^{d}}F(l(n,x))\right) \right)
\leq E\left( \exp \left( -\sum_{y\in \mathcal{L}_{\lambda }^{(n)}}F\left( 
\widetilde{l}(n,y)\right) \right) \right) }\hspace{.5in}&& \\
&=&E_{Q_{n,\lambda }^{(n)}}\left( \exp \left( -\det (B_{\lambda a_{n}})
\int_{\mathbb{T}}F\left( \frac{n}{\det (B_{\lambda a_{n}})}f(x)\right)
dx\right) \right) .
\end{eqnarray*}
Next, follow the line of reasoning used to prove the  Corollary of Theorem 6 
in \cite{Donsker1979}, using the large
deviation upper bound in $\mathcal{L}_1(\mathbb{T})$ and Varadhan's lemma.
From the (lower bound part of)  the  scaling assumption 
and the regular variation property of $%
\det B_{a_{n}},$ we have for any parameter $\lambda >0,$%
\begin{equation*}
\lim \inf_{n\rightarrow \infty }\frac{a_{n}\det (B_{\lambda a_{n}})}{n}%
F\left( \frac{n}{\det (B_{\lambda a_{n}})}y\right) \geq \lambda
^{(1-\gamma )\tra E}\widetilde{F}(y),\text{ }y>0.
\end{equation*}%
Setting  $\mathbf D_n=\det (B_{\lambda a_{n}})$, we obtain
\begin{eqnarray*}
\lefteqn{\lim \sup_{n\rightarrow \infty }\frac{a_{n}}{n}\log E\left( \exp \left(
-\sum_{x\in 
\mathbb{Z}
^{d}}F(l(n,x))\right) \right)}  && \\
&\leq &\lim \sup_{n\rightarrow \infty }\frac{a_{n}}{n}\log
E_{Q_{n,\lambda}^{(n)}}\left( \exp \left(- \mathbf D_n \int_{\mathbb{T}%
}F\left( \frac{n}{\mathbf D_n}f(x)\right) dx)\right)\right) \\
&= &\lim \sup_{n\rightarrow \infty }\frac{a_{n}}{n}\log
E_{Q_{n,\lambda}^{(n)}}\left( \exp \left( -\frac{n}{a_{n}} \int_{\mathbb{T}}%
\frac{a_{n}\mathbf D_n}{n}F\left( \frac{n}{\mathbf D_n}f(x)\right) dx\right) \right) \\
&\leq &-\inf_{f\in \mathcal{F}_{\mathbb{T}}}\left\{ \lambda ^{-1}\mathcal{E}%
_{\widetilde{\eta }}(\sqrt{f},\sqrt{f})+\lambda ^{(1-\gamma )\tra E}\int_{%
\mathbb{T}}\widetilde{F}(f(x))dx\right\} .
\end{eqnarray*}%
The last step comes from Varadhan's lemma. Since the choice of parameter $%
\lambda $ is arbitrary, we can optimize over all $\lambda >0.$
\end{proof}

The following lemma is proved in the appendix. It shows that the constants
appearing in the upper and lower bounds actually match up. In particular,
since this constant appears as both a $\sup $ and an $\inf $ of some
nonnegative quantities, it follows clearly that the constant $k(\eta,%
\widetilde{F})$ defined below takes value in $(0,\infty ).$

\begin{lem}
\label{Constant}Suppose $\widetilde{F}$ is a homogeneous function with
exponent $\gamma \in \lbrack 0,1],$ that is $\widetilde{F}(0)=0,\widetilde{F}%
(y)=\widetilde{F}(1)y^{\gamma }$ for $y>0;$ and $\eta $ is a full
operator-stable law with exponent $E.$ Then there exists a constant $k(\eta ,%
\widetilde{F})\in (0,\infty )$ such that 
\begin{eqnarray*}
k(\eta ,\widetilde{F}) &=&\sup_{\lambda >0}\inf_{f\in \mathcal{F}_{\mathbb{T}%
}}\left\{ \lambda ^{-1}\mathcal{E}_{\widetilde{\eta }}(\sqrt{f},\sqrt{f}%
)+\lambda ^{(1-\gamma )\tra E}\int_{\mathbb{T}}\widetilde{F}(f(x))dx\right\}
\\
&=&\inf_{\Omega \in \mathcal{G}}\inf_{f\in \mathcal{F}_{\Omega }}\left\{ 
\mathcal{E}_{\eta }(\sqrt{f},\sqrt{f})+\int_{\Omega }\widetilde{F}%
(f(x))dx\right\} .
\end{eqnarray*}
\end{lem}

\subsection{Proof of the large deviation principle in $L^1$}

\label{LDPproof}

In this section we indicate how to adapt \cite{Donsker1979} to prove the
large deviation principles as stated in Theorems \ref{TorusLDP} and \ref%
{DirichletLDP}. For this purpose we first develop a large deviation principle 
in the weak topology following Lemma 3.1 and Appendix A in \cite{GKZ}. 

Throughout this subsection we assume 
\begin{equation}
B_{n}^{-1}\mu ^{(n)}\Longrightarrow \eta .  \tag{C-$B_{n}$}
\end{equation}

First we establish asymptotics for  exponential moment generating functions.
Compare with \cite[Lemma A.1]{GKZ} which treats simple random walk on $%
\mathbb{Z}^d$.

\begin{pro}
\label{pro-moment} For the projected occupation measure, for any $f\in C(%
\mathbb{T})$ and any sequence $(a_{n})$ satisfying $a_{n}\rightarrow \infty $
and $a_{n}=o(n)\text{ as }n\rightarrow \infty$,  
\begin{eqnarray*}
\lefteqn{\lim_{n\rightarrow \infty }\frac{1}{n/a_{n}}\log E\left( \exp
\left( \frac{n}{a_{n}}<f,\widetilde{L}_{n}^{(n)}>\right) \right) } \\
&=&\sup_{g\in \mathcal{F}_{\mathbb{T}}}\left\{ \int_{\mathbb{T}%
}f(x)g(x)dx-\lambda ^{-1}\mathcal{E}_{\widetilde{\eta }}(\sqrt{g},\sqrt{g}%
)\right\} .
\end{eqnarray*}%
For the occupation measure with Dirichlet boundary condition, for any
function $f\in C_0(\Omega ),$%
\begin{eqnarray*}
\lefteqn{\lim_{n\rightarrow \infty }\frac{1}{n/a_{n}}\log E\left( \exp
\left( \frac{n}{a_{n}}<f,L_{n}^{(n)}>\right) \mathbf 1_{\{\mbox{%
\scriptsize\em supp}(L_n^{(n)})\subset\Omega\}} \right) } \\
&=&\sup_{g\in \mathcal{F}_{\Omega }}\left\{ \int_{\Omega }f(x)g(x)dx-%
\mathcal{E}_{\eta }(\sqrt{g},\sqrt{g})\right\} .
\end{eqnarray*}
\end{pro}

\begin{proof}

In the lower bound direction, we have the following Feynman-Kac estimates as consequences of functional
limit theorem. For the proof, adapt the arguments given in \cite[Theorem 7.1]{Chen2010}%
. For any sequence $(a_{n})$ satisfying $a_{n}\rightarrow \infty $
and $a_{n}=o(n)\text{ as }n\rightarrow \infty$, 
and any $f\in C(\mathbb{T}),$%
\begin{eqnarray*}
\lefteqn{\lim \inf_{n\rightarrow \infty }\frac{1}{n/a_{n}}\log E\left( \exp
\left( \frac{1}{a_{n}}\sum_{k=1}^{n}f(\widetilde{S}_{k}^{(n)})\right)
\right) } \\
&\geq &\sup_{g\in \mathcal{F}_{\mathbb{T}}}\left\{ \int_{\mathbb{T}%
}f(x)g(x)dx-\lambda ^{-1}\mathcal{E}_{\widetilde{\eta }}(\sqrt{g},\sqrt{g}%
)\right\} .
\end{eqnarray*}%
Similarly, for $f\in C_0(\Omega ),$ 
\begin{eqnarray*}
\lefteqn{\lim \inf_{n\rightarrow \infty }\frac{1}{n/a_{n}}\log E\left( \exp
\left( \frac{1}{a_{n}}\sum_{k=1}^{n}f(S_{k}^{(n)})\right) 
\mathbf 1_{\{\mbox{\scriptsize supp}(L_n^{(n)})\subset\Omega\}}
\right) } \\
&\geq &\sup_{g\in \mathcal{F}_{\Omega }}\left\{ \int_{\Omega }f(x)g(x)dx-%
\mathcal{E}_{\eta }(\sqrt{g},\sqrt{g})\right\} .
\end{eqnarray*}

In the upper bound direction, as a consequence of the convergence assumption, we can adapt the proof of \cite[Theorem 3]{Donsker1979} to have the following large deviation upper bound. 
Let $C$ be a closed of $\mathcal{M}_{1}(\mathbb{T})$. Then 
\begin{equation*}
\limsup_{n\rightarrow \infty }\frac{1}{n/{a_{n}}}\log \widetilde{P}%
_{n,\lambda }^{(n)}(C)\leq -\lambda ^{-1}\inf_{\nu \in C}I_{L_{\widetilde{%
\eta }}}(\nu ),
\end{equation*}%
where 
\begin{equation*}
I_{L_{\widetilde{\eta }}}(\nu )=-\inf_{u\in \mathcal{U}_{\mathbb{T}}}\int_{%
\mathbb{T}}\frac{L_{\widetilde{\eta }}u}{u}(x)d\nu (x).
\end{equation*}%
Similarly, if $C$ is compact in $\mathcal{M}_{1}(\overline{\Omega} ),$ 
\begin{equation*}
\limsup_{n\rightarrow \infty }\frac{1}{n/{a_{n}}}\log P_{n,\Omega
}^{(n)}(C)\leq -\inf_{\nu \in C}I_{L_{\eta }}(\nu ),
\end{equation*}%
where 
\begin{equation*}
I_{L_{\eta }}(\nu )=-\inf_{u\in \mathcal{U}_{\Omega }}\int_{\Omega }\frac{%
L_{\eta }u}{u}(x)d\nu (x).
\end{equation*}%
Either on $\mathbb T$ or in $\Omega$, apply Varadhan's lemma 
(\cite[Theorem 2.2]{Vb}) to the large deviation upper bound to obtain 
the upper bounds needed for Proposition \ref{pro-moment}.
\end{proof}

By the Gartner-Ellis theorem (e.g., \cite[Theorem 4.5.20]{Demb}), we obtain
the large deviation principle in the weak topology stated in the following
Theorem. Compare with \cite[Lemma 3.1]{GKZ}.

\begin{theo}
For any Borel set $B$ in $\mathcal{M}_{1}(\mathbb{T})$ and
any sequence $(a_{n})$ satisfying $a_{n}\rightarrow \infty $
and $a_{n}=o(n)\text{ as }n\rightarrow \infty$,  
\begin{eqnarray*}
-\lambda ^{-1}\inf_{f\in B^{\circ }}I_{L_{\widetilde{\eta }}}(f) &\leq
&\liminf_{n\rightarrow \infty }\frac{1}{n/{a_{n}}}\log \widetilde{P}%
_{n,\lambda }^{(n)}(B) \\
&\leq &\limsup_{n\rightarrow \infty }\frac{1}{n/{a_{n}}}\log \widetilde{P}%
_{n,\lambda }^{(n)}(B)\leq -\lambda ^{-1}\inf_{f\in \overline{B}}I_{L_{%
\widetilde{\eta }}}(f).
\end{eqnarray*}%
Similarly, for any Borel set $A$ in $\mathcal{M}_{1}(\overline{\Omega} ),$ 
\begin{eqnarray*}
-\inf_{f\in A^{\circ }}I_{L_{\eta }}(f) &\leq &\liminf_{n\rightarrow \infty }%
\frac{1}{n/{a_{n}}}\log P_{n,\Omega }^{(n)}(A) \\
&\leq &\limsup_{n\rightarrow \infty }\frac{1}{n/{a_{n}}}\log P_{n,\Omega
}^{(n)}(A)\leq -\inf_{f\in \overline{A}}I_{L_{\eta }}(\eta ).
\end{eqnarray*}
\end{theo}

\bigskip

Next, following \cite{Donsker1979}, we use the local limit theorem to
upgrade the large deviation principle in the weak topology to a result in
the strong $L^{1}$-topology. This is a rather technical task. As shown in 
\cite[Theorem 6]{Donsker1979}, the key point is to obtain a
super-exponential estimate on the $L^{1}$-distance of the density function
to its smooth mollification. This, in turn, requires uniform properties of
the transition probabilities that are provided by the local limit theorem.

Let $\{\psi _{\epsilon }\}$, $\epsilon \rightarrow 0$, be an approximation
of the identity on $%
\mathbb{R}
^{d}$ with $\psi _{\epsilon }$ is smooth, symmetric, compactly supported
inside $(-\frac{\epsilon }{2},\frac{\epsilon }{2})^{d}.$ Thinking of $\psi
_{\epsilon }$ also as a function on $\mathbb{T},$ set $K_{\epsilon }:L^{1}(%
\mathbb{T})\rightarrow L^{1}(\mathbb{T})$ as%
\begin{equation*}
K_{\epsilon }f(x)=\int_{\mathbb{T}}f(y)\psi _{\epsilon }(x-y)dy.
\end{equation*}

\begin{theo}
\label{superexponential}For every $\delta >0,$ $\lambda >0$ and
sequence $a_n$ tending to infinity such that  $a_{n}\left\vert \det
B_{a_{n}}\right\vert \leq n$, we have 
\begin{equation*}
\lim \sup_{\epsilon \rightarrow 0}\lim \sup_{n\rightarrow \infty }\frac{1}{%
n/a_{n}}\log Q_{n,\lambda }^{(n)}\left( f:\int_{\mathbb{T}}\left\vert
K_{\epsilon }f(x)-f(x)\right\vert dx\geq \delta \right) =-\infty .
\end{equation*}
\end{theo}

We follow step by step the proof of \cite[Theorem 5]{Donsker1979}. To this
end, we adapt to our situation the sequence of lemmas in \cite{Donsker1979}
that are used to prove this theorem. The first lemma provides an elementary
way to select a $\delta $-net of functions. Recall that 
\begin{equation*}
\chi _{n}(x)=|\det B_{\left\lfloor \lambda a_{n}\right\rfloor }|^{-1}\chi
_{B_{\left\lfloor \lambda a_{n}\right\rfloor }^{-1}\left( [-\frac{1}{2},%
\frac{1}{2})^{d}\right)} .
\end{equation*}

\begin{lem}
Let $M_{n,\epsilon }\subset C(\mathbb{T})$ be the set of functions 
\begin{equation*}
M_{n,\epsilon }=\{V=(K_{\epsilon }-I)\chi _{n}g:g\in C(\mathbb{T}%
),\left\Vert g\right\Vert _{\infty }\leq 1\}.
\end{equation*}%
For any $\delta >0,$ there exist functions $V_{1},...,V_{J}$ such that for
any $V\in M_{n,\epsilon },$%
\begin{equation*}
\inf_{1\leq i\leq J}\sup_{x\in \mathcal{L}^{(n)}}\left\vert
V(x)-V_{i}(x)\right\vert \leq \frac{\delta }{2},
\end{equation*}%
and 
\begin{equation*}
J=J(n,\epsilon ,\delta )\leq \left( \frac{8}{\delta }+1\right) ^{\left\vert
\det B_{\left\lfloor \lambda a_{n}\right\rfloor }\right\vert }.
\end{equation*}
\end{lem}

\begin{proof}
 The proof is identical to \cite[Lemma 4.1 ]{Donsker1979}.
\end{proof}

The second lemma is similar to \cite[Lemma 4.2]{Donsker1979} and concerns
the uniform control of the transition probabilities. Such uniform control
appears as Assumption (U) in \cite[Section 4.1]{Stroock} and \cite[Section
6.3]{Demb} to obtain the large deviation principle in $\mathcal{L}_1$.

\begin{lem}
\label{lem-DVharnack} There exists $n_{0}\in 
\mathbb{N}
$ and constant $c<\infty $ such that for all $n\geq n_{0},$ 
\begin{equation*}
\sup_{x\in \mathcal{L}^{(n)}}\widetilde{\mu }_{n}^{\ast a_{n}}(x)\leq
c\inf_{x\in \mathcal{L}^{(n)}}\widetilde{\mu }_{n}^{\ast a_{n}}(x).
\end{equation*}
\end{lem}

\begin{proof}
Recall the local limit theorem in \cite{Griffin1986},%
\begin{equation*}
\lim \sup_{n\rightarrow \infty }\sup_{x\in 
\mathbb{Z}
^{d}}\left\vert \det B_{n}\right\vert \cdot \left\vert \mu
^{(n)}(x)-\left\vert \det B_{n}^{-1}\right\vert g(B_{n}^{-1}x)\right\vert =0.
\end{equation*}%
Applying  this local limit result along the sequence $\{\left\lfloor \lambda
a_{n}\right\rfloor \}$ and projecting onto $\mathbb{T}$, we have%
\begin{equation*}
\lim \sup_{n\rightarrow \infty }\sup_{y\in \mathcal{L}^{(n)}}\left\vert
\left\vert \det B_{a_{n}}\right\vert \widetilde{\mu }_{n}^{\ast a_{n}}(y)-%
\widetilde{g}_{\lambda ^{-1}}(y)\right\vert =0.
\end{equation*}
Since the density $g$ is continuous, the desired result follows. 
\end{proof}

We have the following uniform estimate with respect to the starting point.

\begin{lem}
\label{lem1} Let $x$ and $y$ be any two points in $\mathcal{L}_{\lambda
}^{(n)}.$ There exists an integer $n_{0}$ such that for any $n\geq n_{0}$
and any $\theta >0,$%
\begin{equation*}
E_{y}\left[ \exp \left( \theta \sum_{k=1}^{n}V(\widetilde{S}%
_{k}^{(n)})\right) \right] \leq C_{n,\theta }E_{x}\left[ \exp \left( \theta
\sum_{k=1}^{n}V(\widetilde{S}_{k}^{(n)})\right) \right] ,
\end{equation*}%
where $C_{n,\theta }=c\exp (4\theta a_{n})$ and $c$ and $n_{0}$ are as in 
\emph{Lemma 4.2}.
\end{lem}

\begin{proof}
This follows from Lemma \ref{lem-DVharnack} as in the proof of 
\cite[Lemma 4.3]{Donsker1979}.
\end{proof}

\begin{lem}
\bigskip There exists an integer $n_{0}$ such that for any $n\geq n_{0}$ and
any $\theta >0,$%
\begin{equation*}
E_{0}\left[ \exp \left( \theta \sum_{k=1}^{n}V(\widetilde{S}%
_{k}^{(n)})\right) \right] \leq C_{n,\theta }\exp \left( n\widetilde{I}%
_{n}^{\ast }(\theta V)\right) ,
\end{equation*}%
where $C_{n,\theta }$ is as in \emph{Lemma \ref{lem1}} and $\widetilde{I}%
_{n}^{\ast }$ is Legendre transform of 
\begin{equation*}
\widetilde{I}_{n}(\eta )=-\inf_{u\in \mathcal{U}_{\mathbb{T}}}\int_{_{%
\mathbb{T}}}\log \frac{\widetilde{\mu }_{n}u}{u}d\eta ,
\end{equation*}%
that is, 
\begin{equation*}
\widetilde{I}_{n}^{\ast }(\theta V)=\sup_{\eta \in \mathcal{M}_{1}(\mathbb{T}%
)}\left\{ \int_{_{\mathbb{T}}}\theta Vd\eta -\widetilde{I}_{n}(\eta
)\right\} .
\end{equation*}
\end{lem}

\begin{proof}
 Follow \cite[Lemma 4.4]{Donsker1979}.
\end{proof}

Finally, we need the following technical lemma that controls error terms as $%
n\rightarrow \infty$.

\begin{lem}
Let $\eta $ be a probability measure on $\mathcal{L}^{(n)}$ such that $%
\widetilde{I}_{n}(\eta )\leq \frac{\sigma }{a_{n}},$ where $\sigma >0.$ Let $%
V=(K_{\epsilon }-I)\chi _{n}g$ where $g\in C(\mathbb{T}),\left\Vert
g\right\Vert _{\infty }\leq B.$ Then for any $t>0,$%
\begin{equation*}
\int_{_{\mathbb{T}}}Vd\eta \leq B[2h(t\sigma )+2\Delta
_{t}(n)+k_{t}(\epsilon )],
\end{equation*}%
where 
\begin{equation*}
h(l):=2\inf_{a>0}\frac{l+a-\log (1+a)}{a}
\end{equation*}%
and%
\begin{equation*}
\Delta _{t}(n)=\int_{_{\mathbb{T}}}\left\vert \chi _{n}\ast \widetilde{\mu }%
_{n}^{(\left\lfloor ta_{n}\right\rfloor )}(x)-\widetilde{g}_{t/\lambda
}(x)\right\vert dx,
\end{equation*}%
\begin{equation*}
k_{t}(\epsilon )=\sup_{y\in (-\frac{\epsilon }{2},\frac{\epsilon }{2}%
)^{d}}\int_{_{\mathbb{T}}}\left\vert \widetilde{g}_{t/\lambda }(x-y)-%
\widetilde{g}_{t/\lambda }(x)\right\vert dx.
\end{equation*}%
Moreover, we have that $h(l)\rightarrow 0$ as $l\rightarrow 0$ and, for any
fixed $t>0,$ $\Delta _{t}(n)\rightarrow 0$ as $n\rightarrow \infty $ and $%
k_{t}(\epsilon )\rightarrow 0$ as $\epsilon \rightarrow 0.$
\end{lem}

\begin{proof}
 See \cite[Lemma 4.5]{Donsker1979}
\end{proof}

With these five lemmas, the line of reasoning used in \cite[Theorem 5]%
{Donsker1979} gives us Theorem \ref{superexponential}.

\bigskip

\section{Appendix}

\subsection{Proof of regular variation properties}

In this subsection we deduce from the regular variation of the sequence $%
(B_{n})$ the properties of $\widetilde{F}$ stated in the technical
proposition \ref{Homogenous}. We follow \cite{Biskup2001} closely.

\begin{proof}[Proof of {\em Proposition \ref{Homogenous}}]

By Theorem \ref{RegularVariation} (i.e.,
\cite[Theorem 1.10.19 ]{Hazod2001}), under the convergence
assumption $
B_{n}^{-1}\mu ^{(n)}\Longrightarrow \eta $,
there exists a modified normalization sequence $(B_{n}^{\prime
})$ with $B'_n=B_{n}S_{n}$,  $S_{n}\in \mbox{Ivn}(\eta )$,
such that $(B_{n}^{\prime })$ has the regular variation property 
\begin{equation*}
B_{n}^{\prime }(B_{\left\lfloor nt\right\rfloor }^{\prime })^{-1}\rightarrow
t^{-E}
\end{equation*}%
where the convergence is uniform in $t$ on compact subsets of $%
\mathbb{R}
_{+}^{\times }.$ Since $S_{n}\in \mbox{Ivn}(\eta )$ we have \begin{equation*}
(B_{n}^{\prime })^{-1}\mu ^{(n)}\Longrightarrow \eta .\end{equation*}%
Further, since $\mbox{Inv}(\eta )$ is a
compact group, we must have $\det S_{n}=1,$ $\det B_{n}^{\prime }=\det B_{n}.$
Hence we can replace $B_{n}$ in the scaling assumption by $B_{n}^{\prime }$ and
 we have 
\begin{equation*}
\lim_{n\rightarrow \infty }\frac{a_{n}\det (B_{a_{n}}^{\prime })}{n}F\left( 
\frac{n}{\det (B_{a_{n}}^{\prime })}y\right) =\widetilde{F}(y),
\end{equation*}%
uniformly over compact sets in $(0,\infty ).$

Set%
\begin{equation*}
\widetilde{F}_{n}(y):=\frac{a_{n}\det (B_{a_{n}}^{\prime })}{n}F\left( \frac{%
n}{\det (B_{a_{n}}^{\prime })}y\right).
\end{equation*}%
The scaling assumption now reads $\lim_{n\rightarrow \infty }\widetilde{F}%
_{n}(y)=\widetilde{F}(y).$ Note that by assumption, $F:[0,\infty
)\rightarrow \lbrack 0,\infty )$ is a concave, increasing function with $%
F(0)=0,$ therefore both $\widetilde{F}_{n}$ and $\widetilde{F}$ are concave,
non-decreasing, not identically zero with value 0 at 0. Hence $\widetilde{F}%
_{n}$ and $\widetilde{F}$ are continuous and strictly positive in $(0,\infty
),$ and by concavity, $y\rightarrow \frac{\widetilde{F}_{n}(y)%
}{y}$ and $y\rightarrow \frac{\widetilde{F}(y)}{y}$ are both non-increasing
functions.

Now we show that for any $\lambda \in (0,1),$ $\frac{a_{\left\lfloor
\lambda n\right\rfloor }}{a_{n}}$ tends to a finite non-zero limit as $%
n\rightarrow \infty $ $.$ Fix a $y>0$ and write 
\begin{eqnarray*}
\widetilde{F}_{\left\lfloor \lambda n\right\rfloor }(y) &=&\frac{%
a_{\left\lfloor \lambda n\right\rfloor }\det (B_{a_{\left\lfloor \lambda
n\right\rfloor }}^{\prime })}{\left\lfloor \lambda n\right\rfloor }F\left( 
\frac{\left\lfloor \lambda n\right\rfloor }{\det (B_{a_{\left\lfloor \lambda
n\right\rfloor }}^{\prime })}y\right) \\
&=&\frac{a_{\left\lfloor \lambda n\right\rfloor }\det (B_{a_{\left\lfloor
\lambda n\right\rfloor }}^{\prime })}{\left\lfloor \lambda n\right\rfloor }%
F\left( \frac{n}{\det (B_{a_{n}}^{\prime })} \frac{\left\lfloor \lambda
n\right\rfloor /n}{\det (B_{a_{\left\lfloor \lambda n\right\rfloor
}}^{\prime })/\det (B_{a_{n}}^{\prime })}y\right) \\
&=&\frac{a_{\left\lfloor \lambda n\right\rfloor }}{a_{n}}\frac{\det
(B_{a_{\left\lfloor \lambda n\right\rfloor }}^{\prime })/\det
(B_{a_{n}}^{\prime })}{\left\lfloor \lambda n\right\rfloor /n} 
\widetilde{F}_{n}\left( \frac{\left\lfloor \lambda n\right\rfloor /n}{\det
(B_{a_{\left\lfloor \lambda n\right\rfloor }}^{\prime })/\det
(B_{a_{n}}^{\prime })}y\right) .
\end{eqnarray*}%

As $(a_{n})$ is an increasing
sequence and $\det (B_{n})$ is non-decreasing with respect to $n$, we have
$\det (B_{a_{\left\lfloor \lambda n\right\rfloor }}^{\prime })/\det
(B_{a_{n}}^{\prime })\leq 1$ 
 Since $%
y\rightarrow \frac{\widetilde{F}_{n}(y)}{y}$ is non-increasing, we have%
\begin{equation*}
\frac{\widetilde{F}_{n}\left( \frac{\left\lfloor \lambda n\right\rfloor /n}{%
\det (B_{a_{\left\lfloor \lambda n\right\rfloor }}^{\prime })/\det
(B_{a_{n}}^{\prime })}y\right) }{\frac{\left\lfloor \lambda n\right\rfloor /n%
}{\det (B_{a_{\left\lfloor \lambda n\right\rfloor }}^{\prime })/\det
(B_{a_{n}}^{\prime })}y}\leq \frac{\widetilde{F}_{n}(\frac{\left\lfloor
\lambda n\right\rfloor }{n} y)}{\frac{\left\lfloor \lambda
n\right\rfloor }{n} y}.
\end{equation*}%
Therefore%
\begin{equation*}
\widetilde{F}_{\left\lfloor \lambda n\right\rfloor }(y)\leq \frac{%
a_{\left\lfloor \lambda n\right\rfloor }}{a_{n}} \frac{\widetilde{F}%
_{n}\left( \frac{\left\lfloor \lambda n\right\rfloor }{n} y\right) }{%
\frac{\left\lfloor \lambda n\right\rfloor }{n}}.
\end{equation*}%
Letting $n\rightarrow \infty $ on both sides, the scaling assumption yields 
\begin{equation*}
\widetilde{F}(y)\leq \lim \inf_{n\rightarrow \infty }\frac{a_{\left\lfloor
\lambda n\right\rfloor }}{a_{n}}\frac{1}{\lambda }\widetilde{F}%
(\lambda y).
\end{equation*}%
Therefore 
\begin{equation*}
\lim \inf_{n\rightarrow \infty }\frac{a_{\left\lfloor \lambda n\right\rfloor
}}{a_{n}}\geq \frac{\lambda \widetilde{F}(y)}{\widetilde{F}(\lambda y)}.
\end{equation*}

\bigskip

Since $(a_{n})$ is an increasing sequence by assumption, we have 
\begin{equation*}
\frac{\lambda \widetilde{F}(y)}{\widetilde{F}(\lambda y)}\leq \lim
\inf_{n\rightarrow \infty }\frac{a_{\left\lfloor \lambda n\right\rfloor }}{%
a_{n}}\leq \lim \sup_{n\rightarrow \infty }\frac{a_{\left\lfloor \lambda
n\right\rfloor }}{a_{n}}\leq 1.
\end{equation*}

Replacing $\lambda $ by $\frac{1}{\lambda },$ we conclude that for all $%
\lambda \in (0,\infty ),$ $\frac{a_{\left\lfloor \lambda n\right\rfloor }}{%
a_{n}}$ is uniformly bounded away from 0 and uniformly bounded from above.

\bigskip

Let $\phi (\lambda )$ be defined for each $\lambda \in (0,\infty )$ as a
sub-sequential limit of $\frac{a_{\left\lfloor \lambda n\right\rfloor }}{a_{n}%
}.$ Namely, choose some ($\lambda $-dependent) sub-sequence $t_{n}\rightarrow
\infty $ and  set $\phi (\lambda )=\lim_{n\rightarrow \infty }\frac{%
a_{\left\lfloor \lambda t_{n}\right\rfloor }}{a_{t_{n}}}.$ From the above
reasoning we know that $\phi (\lambda )\in (0,\infty ).$ Consider the equation%
\begin{equation*}
\widetilde{F}_{\left\lfloor \lambda n\right\rfloor }(y)=\frac{%
a_{\left\lfloor \lambda n\right\rfloor }}{a_{n}} \frac{\det
(B_{a_{\left\lfloor \lambda n\right\rfloor }}^{\prime })/\det
(B_{a_{n}}^{\prime })}{\left\lfloor \lambda n\right\rfloor /n} 
\widetilde{F}_{n}\left( \frac{\left\lfloor \lambda n\right\rfloor /n}{\det
(B_{a_{\left\lfloor \lambda n\right\rfloor }}^{\prime })/\det
(B_{a_{n}}^{\prime })}y\right) ,
\end{equation*}%
and take the limit along the sub-sequence $(t_{n}).$ 
We can indeed take the limit on the right hand side of the equation because 
$\widetilde{F}$ is continuous, and the convergences   $\widetilde{F}%
_{n}(y)\rightarrow \widetilde{F}(y)$ and $B_{n}^{\prime }(B_{\left\lfloor
nt\right\rfloor }^{\prime })^{-1}\rightarrow t^{-E}$ are uniform over
compact sets. This yields 
\begin{equation*}
\widetilde{F}(y)=\phi (\lambda ) \frac{\phi (\lambda )^{\tra E}}{\lambda }%
 \widetilde{F}\left( \frac{\lambda }{\phi (\lambda )^{\tra E}}y\right) .
\end{equation*}%
Note that the function 
$$z\rightarrow \frac{z^{\tra E}}{\lambda } 
\widetilde{F}\left( \frac{\lambda }{z^{\tra E}}y\right) $$ is non-decreasing
because $y\rightarrow \frac{\widetilde{F}(y)}{y}$ is non-increasing. As $%
z\rightarrow \frac{\widetilde{F}(y)}{z}$ is strictly decreasing, the solution
$z_0=z(\lambda,y)$ to 
\begin{equation*}
\frac{\widetilde{F}(y)}{z}=\frac{z^{\tra E}}{\lambda }\cdot \widetilde{F}\left( 
\frac{\lambda }{z^{\tra E}}y\right)
\end{equation*}%
is unique. Hence the limit $\phi (\lambda )=\lim_{n\rightarrow \infty }\frac{%
a_{\left\lfloor \lambda n\right\rfloor }}{a_{n}}$ exists in $(0,\infty )$
for all $\lambda \in (0,\infty ).$

Observe that  
\begin{eqnarray*}
\phi (\lambda _{1}\lambda _{2}) &=&\lim_{n\rightarrow \infty }\frac{%
a_{\left\lfloor \lambda _{1}\lambda _{2}n\right\rfloor }}{a_{n}} \\
&=&\lim_{n\rightarrow \infty }\frac{a_{\left\lfloor \lambda _{1}\lambda
_{2}n\right\rfloor }}{a_{\left\lfloor \lambda _{2}n\right\rfloor }}\cdot 
\frac{a_{\left\lfloor \lambda _{2}n\right\rfloor }}{a_{\left\lfloor
n\right\rfloor }} \\
&=&\phi (\lambda _{1})\phi (\lambda _{2}).
\end{eqnarray*}%
Therefore $\phi$ is multiplicative and $
\phi (\lambda )=\lambda ^{\kappa }$
with $\kappa =\log _{2}\phi (2).$ Plugging this back in 
\begin{equation*}
\widetilde{F}(y)=\phi (\lambda )\cdot \frac{\phi (\lambda )^{\tra E}}{\lambda }%
\cdot \widetilde{F}\left( \frac{\lambda }{\phi (\lambda )^{\tra E}}y\right) ,
\end{equation*}%
we have%
\begin{equation*}
\widetilde{F}(y)=\lambda ^{\kappa }\cdot \frac{\lambda ^{\kappa \tra E}}{%
\lambda }\cdot \widetilde{F}\left( \frac{\lambda }{\lambda ^{\kappa \tra E}}%
y\right) .
\end{equation*}%
Setting $y=1$ gives
\begin{equation*}
\widetilde{F}(1)=\lambda ^{\kappa +\kappa \tra E-1}\cdot \widetilde{F}(\lambda
^{1-\kappa \tra E}),
\end{equation*}%
so that
\begin{equation*}
\widetilde{F}(y)=\widetilde{F}(1)y^{\frac{1-\kappa \tra E-\kappa }{1-\kappa \tra E}%
}.
\end{equation*}

The fact that
\begin{equation*}
\lim_{n\rightarrow \infty }\frac{\log a_{n}}{\log n}=\kappa ,
\end{equation*}%
follows exactly from the reasoning in \cite{Biskup2001}. 
\end{proof}

\subsection{Discussion of the constant $k(\protect\eta,\widetilde{F})$ of
Lemma \protect\ref{Constant}}

In this subsection, we follow the truncation argument in \cite{Donsker1975}
to prove Lemma \ref{Constant}. With the notation of Section \ref{sec-DV},
let 
\begin{equation*}
J:=\sup_{\lambda >0}\inf_{f\in \mathcal{F}_{\mathbb{T}}}\left\{ \lambda ^{-1}%
\mathcal{E}_{\widetilde{\eta }}\left( \sqrt{f},\sqrt{f}\right) +\lambda
^{(1-\gamma )\tra E}\int_{\mathbb{T}}\widetilde{F}(f(x))dx\right\}
\end{equation*}%
be the upper bound appearing in Proposition \ref{UpperBound}. Let $\epsilon
>0$ be an arbitrary small number. To prove Lemma \ref{Constant}, it suffices
to find $\Omega \in \mathcal{G}$ and $g\in \mathcal{F}_{\Omega }$ such that 
\begin{equation*}
\mathcal{E}_{\eta }\left( \sqrt{g},\sqrt{g}\right) +\int_{\Omega }\widetilde{%
F}(g(x))dx\leq J+\epsilon .
\end{equation*}%
For any $\lambda $ (we will choose $\lambda$ large enough later on), by the
definition of $J$, there exists $f\in \mathcal{F}_{\mathbb{T}}$ such that 
\begin{equation*}
\lambda ^{-1}\mathcal{E}_{\widetilde{\eta }}(\sqrt{f},\sqrt{f})+\lambda
^{(1-\gamma )\tra E}\int_{\mathbb{T}}\widetilde{F}(f(x))dx<J+\frac{\epsilon 
}{2}.
\end{equation*}%
We can think of functions on $\mathbb{T}$ also as functions on the
fundamental domain $[0,1)^{d}.$ Following \cite[Lemma 3.4]{Donsker1975}, let 
\begin{equation*}
E^{\lambda }=\bigcup _{i=1}^{d}\left(\left\{0\leq x_{i}\leq \frac{1}{\sqrt[4]%
{\lambda }}\right\}\bigcup \left\{1-\frac{1}{\sqrt[4]{\lambda }}\leq
x_{i}<1\right\}\right).
\end{equation*}%
Note that there exists $a\in \mathbb{T}$ such that the translated function $%
f_{a}(x)=f(x-a) $ satisfies%
\begin{equation*}
\int_{E^{\lambda }}f_{a}dx\leq \frac{2d}{\sqrt[4]{\lambda }}.
\end{equation*}%
Because of translation invariance of the expression on the torus, we can
replace $f$ by $f_{a}$ in the expression without changing the value.
Therefore we may assume that $f\in \mathcal{F}_{\mathbb{T}}$ satisfies 
\begin{equation}  \label{cruxf}
\lambda ^{-1}\mathcal{E}_{\widetilde{\eta }}(\sqrt{f},\sqrt{f})+\lambda
^{(1-\gamma )\tra E}\int_{\mathbb{T}}\widetilde{F}(f(x))dx<J+\frac{\epsilon 
}{2},
\end{equation}%

and%
\begin{equation*}
\int_{E^{\lambda }}fdx\leq \frac{2d}{\sqrt[4]{\lambda }}.
\end{equation*}

Consider a smooth bump function $\phi _{0}$ on $\mathbb{R}$ such that $\phi _{0}=1$ on $[\frac{1}{\sqrt[4]{\lambda}},1-\frac{1}{\sqrt[4]{\lambda}}]$, it vanishes outside $\left(\frac{1}{2\sqrt[4]{\lambda}},1-\frac{1}{2\sqrt[4]{\lambda}}\right)$, and $\left\vert
\bigtriangledown \phi _{0}\right\vert \leq 3 \sqrt[4]%
{\lambda }$.
Let $\widetilde{\phi _{0}}(x_{1},...,x_{d})=\phi _{0}(x_{1})...\phi
_{0}(x_{d})$ and $\psi (x)=\widetilde{\phi _{0}}(x)^{2}.$ Then $\left\Vert
\bigtriangledown \widetilde{\phi _{0}}\right\Vert \leq 3 \sqrt{d}\cdot \sqrt[4]%
{\lambda }.$

Let $T_{\lambda }:=\lambda ^{E}\left( [0,1)^{d}\right) $, that is the image
of the fundamental domain $[0,1)^{d}$ under the transformation $\lambda
^{E}. $ Given a function $h$ defined on $[0,1)^{d},$ let $h_{\lambda }$ be
the function on $T_{\lambda }$ defined by 
\begin{equation*}
h_{\lambda }(x):=\left\vert \det (\lambda ^{-E})\right\vert h(\lambda
^{-E}x).
\end{equation*}

Now, set $\Omega =\lambda ^{E}(0,1)^{d}$ where $\lambda $ is sufficiently
large and 
\begin{equation*}
g(x):=\frac{(f\psi )_{\lambda }(x)}{\int_{\mathbb{R}^{d}}(f\psi )_{\lambda
}(x)dx}.
\end{equation*}%
Then, we claim that 
\begin{equation*}
\mathcal{E}_{\eta }\left( \sqrt{g},\sqrt{g}\right) +\int_{\Omega}\widetilde{F%
}(g(x))dx\leq J+\epsilon .
\end{equation*}%
To see this, first note that $g$ is supported on $T_{\lambda }$ and that, by
the scaling properties $t^{E}(W)=t\cdot W$ of the L\'{e}vy measure, we have 
\begin{equation*}
\mathcal{E}_{\eta }\left( \sqrt{g},\sqrt{g}\right) \leq \lambda ^{-1}%
\mathcal{E}_{\widetilde{\eta }}\left( \sqrt{g_{1/\lambda }},\sqrt{%
g_{1/\lambda }}\right) .
\end{equation*}%
Also, since $\widetilde{F}(y)=\widetilde{F}(1)y^{\gamma }$, we have 
\begin{equation*}
\int_{\Omega}\widetilde{F}(g(x))dx=\lambda ^{(1-\gamma )\tra E}\int_{\mathbb{%
T}}\widetilde{F}(g_{1/\lambda }(x))dx.
\end{equation*}%
Hence we obtain 
\begin{eqnarray*}
\lefteqn{\mathcal{E}_{\eta }\left( \sqrt{g},\sqrt{g}\right) +\int_{\Omega}%
\widetilde{F}(g(x))dx} \\
&\leq &\lambda ^{-1}\mathcal{E}_{\widetilde{\eta }}\left( \sqrt{g_{1/\lambda
}},\sqrt{g_{1/\lambda }}\right) +\lambda ^{(1-\gamma )\tra E}\int_{\mathbb{T}%
}\widetilde{F}(g_{1/\lambda }(x))dx.
\end{eqnarray*}

Since $f$ satisfies (\ref{cruxf}), the choice of the bump function and the
fact that real part of the eigenvalues of $E$ are all $\geq \frac{1}{2}$
guarantee that for $\lambda $ sufficiently large, 
\begin{equation*}
\lambda ^{-1}\mathcal{E}_{\widetilde{\eta }}\left( \sqrt{g_{1/\lambda }},%
\sqrt{g_{1/\lambda }}\right) +\lambda ^{(1-\gamma )\tra E}\int_{\mathbb{T}}%
\widetilde{F}(g_{1/\lambda }(x))dx<J+\epsilon .
\end{equation*}

\subsection{Explicit computation of constants}

In this section, we illustrate Theorem \ref{ReturnAsymptotic} 
concerning switch-walk-switch random walks on certain wreath products $K\wr H$
and give some indications concerning exact the computation 
of the constant $k(\eta,\widetilde{F}_K)$.
Let $\nu $ denote a symmetric probability measure
on the lamp-group $K$ and let $F:(0,\infty)\ra (0,\infty)$ be such that $F(n)=
-\log \nu^{(2n)}(e_K)$. 
Assume that $%
\mu $ is a symmetric measure on the  base-group $H=\mathbb{Z}^{d}$ satisfying
the convergence assumption 
\begin{equation*}
n^{-E}\mu ^{(n)}\Longrightarrow \eta ,
\end{equation*}
that is, $\mu$ is in the domain of normal attraction of $\eta$ where $\eta$ is 
an operator-stable law with exponent $E$.
Set
\begin{equation*}
\hat{\eta}=e^{-\Theta}\;\mbox{ and } \widehat{L_\Theta}f= \Theta \hat{f}.
\end{equation*}%
We will treat cases where $F$ satisfies the scaling assumption (S-$n^E$-$a_n$)
for some sequence $a_n$. Let  $\widetilde{F}$ be the 
corresponding limit function  and $\gamma$ be the associated scaling exponent.
See Definition \ref{def-scal} and Proposition \ref{Homogenous}.

Let $q$ be the switch-walk-switch measure on $G=K\wr \mathbb{Z}^{d}$ given
by $q=\nu \ast \mu \ast \nu $.

\begin{exa}[Power decay on $K$, $\gamma =0$] 
This is the continuation of Example \ref{LampPowerDecay}.
Namely, assume
$cn^{-\theta }\leq \nu ^{(2n)}(e_{K})\leq Cn^{-\theta }$.
This means  that $|F(y) - \theta \log y|\le C'$ and 
$\widetilde{F}(y)=\theta \mathbf 1_{(0,\infty)}(y)$, that is, $\gamma=0$.
Under this hypothesis, Theorem \ref{ReturnAsymptotic} yields 
\begin{equation*}
\frac{1}{(2n)^{\tra E/(\tra E+1)}(\log 2n)^{1/(\tra E+1)}}\log
q^{(2n)}(e)=-k(\eta, \theta\mathbf 1_{(0,\infty)} ).
\end{equation*}%
A simple scaling argument as in \cite{Donsker1975,Donsker1979} 
shows that the constant 
$k(\eta, \theta\mathbf 1_{(0,\infty)} )$ 
as in
Lemma \ref{Constant} can be written as%
\begin{equation}
k(\eta, \theta\mathbf 1_{(0,\infty)} ) =\theta ^{1/(\tra E+1)}(\tra E+1)\left( \frac{\lambda
_{1}(\Theta )}{\tra E}\right) ^{\tra E/(\tra E+1)}.
\label{thetaconstant}
\end{equation}%
Here 
\begin{equation*}
\lambda _{1}(\Theta )=\inf_{B:\left\vert B\right\vert =1}\lambda
_{1}(\Theta ,B)
\end{equation*}%
where the infimum is taken over all bounded open sets $B\subset \mathbb{R}^{d}$ 
such that the Lebesgue measure $\left\vert B\right\vert =1,\left\vert \partial
B\right\vert =0,$ and $\lambda _{1}(\Theta ,B)$ is the principle
eigenvalue for $L_{\Theta }$ with Dirichlet boundary on $B$. 
\end{exa}

\begin{exa}[Nonamenable $K$, $\gamma =1$]
Suppose the lamp group $K$ is nonamenable and 
let $\rho $ denote the spectral radius of $\nu$ so that
\begin{equation*}
\left( \nu ^{(2n)}(e_{K})\right) ^{\frac{1}{2n}}\rightarrow \rho .
\end{equation*}%
By \cite[Theorem 3.16]{Pittet2002}, the 
switch-walk-switch measure $q=\nu \ast \mu \ast \nu $  on 
$G=K\wr \mathbb{Z}^{d}$ has spectral radius $%
\rho ^{2}$, namely, 
\begin{equation*}
\left( q^{(2n)}(\mathbf{e})\right) ^{\frac{1}{2n}}\rightarrow \rho ^{2}.
\end{equation*}%
 
We can recover this result using Theorem \ref{ReturnAsymptotic} 
(actually, a rather trivial special case).
We have  
$$F(y)= (\log\rho^2) y +o(y) \mbox{ and }\widetilde{F}= (\log \rho^2) y.$$ 
The variational problem giving the constant $k(\eta,\widetilde{F})$ becomes 
\begin{eqnarray*}
k(\eta,\widetilde{F}) 
&=&\log \rho^2 +\inf \{\mathcal{E}_{\eta }(f,f):f\geq 0,\left\Vert
f\right\Vert _{2}=1\} \\
&=&\log \rho^2 .
\end{eqnarray*}

Further, if a more precise  local limit theorem is known for the random walk 
driven by $\nu $ on $K$, we can derive a log-limit for 
$\rho^{-4n}q^{(2n)}(\mathbf{e})$. 
For example, assume that
\begin{equation*}
\nu ^{(2n)}(e_{K})\sim c(\nu )n^{-\theta }\rho ^{2n},
\end{equation*}%
for some $\theta>0$. Then we have 
\begin{equation*}
\log \left( \rho ^{-4n}q^{(2n)}(\mathbf{e})\right) \sim -k(\eta, \theta
\mathbf 1_{(0,\infty)}
)(2n)^{\tra E/(\tra E+1)}(\log 2n)^{1/(\tra E+1)}
\end{equation*}
where the constant $k(\eta,\theta \mathbf 1_{(0,\infty)})$ is given by ($\ref{thetaconstant}$).
\end{exa}

\begin{exa}[Cases when $\gamma\in (0,1)$]
The case when $\log \nu ^{(2n)}(e_{K})\sim -cn^{\gamma }$, 
$\gamma \in (0,1)$, 
 presents two different difficulties. First, there are very few examples of group $K$ for which such asymptotics is known (even so, we do produce such examples 
above).  Second, the corresponding variational problem describing 
$k(\eta, s\mapsto c s^\gamma)$ 
is not as well studied and explicit solutions 
are not known except for certain cases. See \cite{Schmidt} where is $\gamma$ is half of our $\gamma$. 

Assume that $K,\nu$ are such that $\log \nu ^{(2n)}(e_{K})\sim -cn^{\gamma }$, 
$\gamma \in (0,1)$ so that $\widetilde{F}(y)=cy^\gamma$. Assume further that
  $H=\mathbb{Z}$ and $\mathcal{E%
}_{\eta }\left( f,f\right) =a\int_{\mathbb{R}}\left\vert \nabla f\right\vert
^{2}dx$. In this case,  \cite[Proposition 5.1]{Schmidt} 
provides an explicit solution for the variational problem describing the constant $k(\eta,\widetilde{F})$ and one obtains
\begin{equation*}
\log q^{(2n)}(\mathbf{e})\sim -
k(\eta ,\widetilde{F})
(2n)^{\frac{1+\gamma }{3-\gamma }%
}
\end{equation*}%
where 
\begin{equation*}
k(\eta ,\widetilde{F})=
c^{\frac{2}{3-\gamma }}(2a)^{\frac{1-\gamma }{3-\gamma }%
}\left( \frac{3-\gamma }{1+\gamma }\right) \left( \frac{\sqrt{\pi }\Gamma
\left( \frac{3-\gamma }{2-2\gamma }\right) }{\Gamma \left( \frac{1}{1-\gamma 
}\right) }\right) ^{\frac{2-2\gamma }{3-\gamma }},
\end{equation*}%
and the minimizer is the function $(\cos \left\vert x\right\vert )^{\frac{1}{%
1-\gamma }}\boldsymbol{1}_{[0, \pi /2]}(|x|)$,
properly dilated and normalized. 

When $H=\mathbb Z^d$, $\mathcal{E%
}_{\eta }\left( f,f\right) =a\int_{\mathbb{R}^d}\left\vert \nabla f\right\vert
^{2}dx$,  and $\gamma $ is $1/2$ --- for example, this is
achieved by the switch-walk-switch random walk on the wreath product 
$\mathbb{Z}_{2}\wr \mathbb{Z}%
^{2}$ --- the minimizer is given by a Bessel function and the constant
$k(\eta,s\mapsto cs^{1/2})$ is
explicitly computable, see \cite[Proposition 5.2]{Schmidt}.
\end{exa}


\providecommand{\bysame}{\leavevmode\hbox to3em{\hrulefill}\thinspace} %
\providecommand{\MR}{\relax\ifhmode\unskip\space\fi MR } 
\providecommand{\MRhref}[2]{  \href{http://www.ams.org/mathscinet-getitem?mr=#1}{#2}
} \providecommand{\href}[2]{#2}

\end{document}